\documentclass[12pt]{article}
\usepackage[margin=0.8in]{geometry}
\usepackage[latin1]{inputenc}
\usepackage{amsmath, amssymb, amsbsy, mathdots, mathabx}
\usepackage{amsthm}
\usepackage{enumitem}
\usepackage[all]{xy}
\usepackage{tikz}
\usepackage{changepage}

\title{Local Base Change via Tate Cohomology}
\author{Niccol\`o Ronchetti}
\date{\today}

\theoremstyle{theorem}
\newtheorem{thm}{Theorem}
\newtheorem{lem}[thm]{Lemma}
\newtheorem{prop}[thm]{Proposition}

\newtheorem{conj}[thm]{Conjecture}
\theoremstyle{definition}
\newtheorem{defn}{Definition}
\newtheorem{exmp}{Example}

\theoremstyle{remark}
\newtheorem*{rem}{Remark}
\renewcommand{\mod}{\ensuremath{\bmod \,}}

\newcommand{\Hom}{\ensuremath{\mathrm{Hom}}}
\newcommand{\Aut}{\ensuremath{\mathrm{Aut}}}

\newcommand{\Gal}{\ensuremath{\mathrm{Gal}}}
\newcommand{\Rep}{\ensuremath{\mathrm{Rep}}}
\newcommand{\Irrep}{\ensuremath{\mathrm{Irrep}}}
\newcommand{\Ind}{\ensuremath{\mathrm{Ind}}}
\newcommand{\cInd}{\ensuremath{\operatorname{c-Ind}}}
\newcommand{\Tr}{\ensuremath{\mathrm{Tr}}}
\newcommand{\supp}{\ensuremath{\mathrm{Supp}}}
\newcommand{\Stab}{\ensuremath{\mathrm{Stab}}}
\newcommand{\Frob}{\ensuremath{\mathrm{Frob}}}

\newcommand{\rk}{\ensuremath{\mathrm{rk}}}
\newcommand{\ch}{\ensuremath{\mathrm{char}}}

\newcommand{\id}{\ensuremath{\mathrm{id}}}

\newcommand{\diagon}{\ensuremath{\mathrm{diag}}} 
\newcommand{\Cl}{\ensuremath{\mathrm{Cl}}} 

\newcommand{\im}{\ensuremath{\mathrm{Im}}}

\newcommand{\res}{\ensuremath{\mathrm{res}} \,}
\newcommand{\lra}{\ensuremath{\longrightarrow}}

\newcommand{\Q}{\ensuremath{\mathbb Q}}
\newcommand{\Qp}{\ensuremath{\mathbb Q_p}}

\newcommand{\Fl}{\ensuremath{\mathbb F_l}}
\newcommand{\FFl}{\ensuremath{\overline { \mathbb F_l}}}
\newcommand{\F}[1]{\ensuremath{\mathbb F_{#1}}}

\newcommand{\Z}{\ensuremath{\mathbb Z}}
\newcommand{\Zl}{\ensuremath{\mathbb Z_l}}
\newcommand{\Ql}{\ensuremath{\mathbb Q_l}}
\newcommand{\C}{\ensuremath{\mathbb C}}

\newcommand{\Gl}[1]{\ensuremath{\mathrm{GL}_{#1} }}
\newcommand{\GL}[2]{\ensuremath{\mathrm{GL}_{#1} (#2) }}
\newcommand{\Norm}[2]{\ensuremath{\mathrm{Norm}_{#1 | #2} }}
\newcommand{\Res}[3]{\ensuremath{\mathrm{Res}_{#1 / #2}  #3 }} 
\newcommand{\OO}{\ensuremath{\mathcal O}}
\newcommand{\Rho}{\mathrm{P}}

\newcommand\blfootnote[1]{%
  \begingroup
  \renewcommand\thefootnote{}\footnote{#1}%
  \addtocounter{footnote}{-1}%
  \endgroup
}

\begin{document}

\maketitle

\section{Introduction}
In this paper we propose a new way to realize cyclic base change for prime degree extensions of characteristic zero local fields; we work with $l$-adic representations (that is, representations on vector spaces defined over $\overline \Ql$ - in fact over some finite extension of $\Ql$) and we restrict our attention to cuspidal representations. Every representation of a $p$-adic group mentioned in this paper is tacitly assumed to be smooth (that is, every vector has an open stabilizer). \blfootnote{2010 \textit{Mathematics Subject Classification.} 11F70, 11S37, 22E50.}
\subsection{Motivation}
The motivation to investigate cyclic base change in this situation comes from two conjectures of Treumann and Venkatesh (section 6.3 in \cite{TV}) where they investigate a particular case of Langlands functoriality, which we now recall.

For every connected reductive group $\mathbf G$ over a $p$-adic field $F$, Langlands conjectured (see for instance \cite{vogan} for a detailed exposition) the existence of finite-to-one maps between (isomorphism classes of) complex irreducible admissible representations of $G= \mathbf G(F)$ and (equivalence classes of) admissible Langlands parameters, which are group homomorphisms $\phi: WD_F \lra {^{L}} \mathbf G = \widehat {\mathbf G} \rtimes \Gal(F)$ satisfying some technical conditions\footnote{Here $WD_F$ denotes the Weil-Deligne group of $F$, a slight technical modification of the Weil group of $F$. The dual group $\widehat{\mathbf G}$ is a complex reductive Lie group.}. This correspondence should be well-behaved with respect to many different constructions, and preserve analytic objects such as $L$-functions and $\varepsilon$-factors that can be independently associated to a $G$-representation as well as to a Langlands parameter. \\
Given now two connected reductive groups $\mathbf G, \mathbf G'$ over $F$ and a well-behaved morphism between their $L$-groups $^{L} \phi: {^{L}} \mathbf G \lra {^{L}} \mathbf G'$ which sends Langlands parameters of $G$ into Langlands parameters of $G'$, the functoriality conjecture (see for example \cite{arthur} and \cite{gelbart}) predicts that the associated correspondence between $G$- and $G'$-representations satisfies many nice properties (for example, it is compatible with $L$-functions). It is then an interesting question to give an explicit description of such correspondence.

Treumann and Venkatesh investigate Langlands functoriality in the following setup: let $\mathbf G$ be a reductive algebraic group defined over a $p$-adic field $F$ and an automorphism $\sigma \in \Aut(\mathbf G)$ of order $l$ whose fixed point set is the algebraic group $\mathbf H = \mathbf G^{\sigma}$. \\
They introduce the notion of linkage\footnote{This is supposed to be a representation-theoretic version of the Brauer homomorphism.}: a $\mod l$ representation $\pi$ of $G^{\sigma}$ is \emph{linked} to a $\sigma$-fixed $\mod l$ representation $\Pi$ of $G$ if the Frobenius twist $\pi^{(l)}$ appears as a Jordan-Holder constituent of the Tate cohomology $T^*(\Pi)$, which is naturally a $G^{\sigma}$-representation (see section 3 and definition 6.2 in \cite{TV}).
Under some technical assumptions, they conjecture that there exists an admissible $L$-homomorphism $^{L} \phi: {^{L}}H \lra {^{L}} G$ such that if $\pi$ is linked to the $\sigma$-invariant $\Pi$, then $^{L} \phi$ sends the Langlands parameter of $\pi$ into the one of $\Pi$. \\
Looking at the correspondence of $H$- and $G$-representations induced by $^{L} \phi$ in the opposite direction, the conjecture says that among the $H$-representations corresponding to $\Pi$ there should be all $\pi$'s whose Frobenius twists $\pi^{(l)}$ are Jordan-Holder factors of the Tate cohomology $T^*(\Pi)$: this is what is meant by the catchy phrase `Tate cohomology realizes Langlands functoriality'. \\

Our main theorem proves the conjecture of Treumann and Venkatesh in the special case where $\mathbf G = \Gl{n}$, $\sigma$ is a Galois automorphism of prime order $l$ and the two representations are cuspidal, of level zero and minimal-maximal type. \\
Notice that the assumption on $\sigma$ means that we are investigating an important special case of Langlands functoriality known as \emph{base change} (see section 2 for a detailed definition), and the explicit description of the base change given via Tate cohomology is quite remarkable.

\subsection{Statement of results}
Let $F \supset E$ be a degree $l$ Galois extension of characteristic zero, non-archimedean local fields of residue characteristic $p$, where $p$ and $l$ are different primes. Let $n$ be a positive integer coprime to both $l$ and $p$.
We say that a representation $\left( \rho, V \right)$ of $\GL{n}{F}$ is $\Gal(F/E)$-equivariant when for each element $\gamma \in \Gal(F/E)$ the representations $\rho$ and $\rho \circ \gamma$ are abstractly isomorphic (here $\gamma$ acts componentwise on $\GL{n}{F}$).
A representation $\pi$ of $\GL{n}{E}$ is said to be of level zero and minimal-maximal type if it is the induction from the (open, compact mod center) subgroup $E^* \GL{n}{\OO_E}$ of a representation inflated from a cuspidal representation of $\GL{n}{k_E}$. \\

The main result is the following. All non-standard notions mentioned in the theorem are explained in section 2. \\
Let $\pi$ be an $l$-adic cuspidal, level zero, minimal-maximal type irreducible representation of $\GL{n}{E}$ and $\left( \rho, V \right)$ be an $l$-adic cuspidal, level zero, minimal-maximal type irreducible representation of $\GL{n}{F}$ which is $\Gal(F/E)$-equivariant.  From $(\rho, V)$ we construct a $\mod l$ representation of $\GL{n}{E}$ as follows: let $L \subset V$ be a $\rho$-stable lattice, and $\Rho: \GL{n}{F} \rtimes \Gal(F / E) \lra V$ be an extension of $\rho$.
Then $L$ is $\Rho ( \gamma)$-stable and the \emph{Tate cohomology} \[ T^0(\rho):= L^{\Rho( \gamma)} / \im \left( 1 + \Rho ( \gamma) + \Rho(\gamma)^2 + \ldots + \Rho(\gamma)^{l-1} \right) \] yields an irreducible $\mod l$ representation of $\GL{n}{E}$ whose isomorphism class is independent of the choices of $L$ and $\Rho$: see lemma \ref{T0welldefined} for the proof of this claim.
\begin{thm}\label{mainthm} With the notation above, the $\mod l$ reductions $r_l \left( \pi \right)$ and $r_l \left( \rho \right)$ (which are irreducible $\overline {\F{l}}$-representations) are in base change (denoted $\mathrm{bc}$) if and only if \[ r_l(\pi)^{(l)} \cong T^0(\rho), \] where on the left we have the Frobenius twist of $r_l(\pi)$ (i.e. the scalar $\lambda$ acts as $\lambda^{1/l}$ - see also formula \ref{deftwist} in the Notation section).

In particular, \[ r_l(\pi)^{(l)} \cong T^0( \mathrm{bc}(\pi)). \]
\end{thm}
Notice that the underlying space of $T^0(\rho)$ is by construction $l$-torsion - where $l$ is the order of the Galois element $\gamma$; thus if we want to compare $T^0(\rho)$ to a reduction modulo a prime $l'$ of an irreducible representation $\pi$ whose coefficient field is, say, a number field, we are forced to pick $l'=l$.
\begin{rem} We will show that if we have $r_l(\rho) \cong \mathrm{bc}(r_l(\pi))$, then the isomorphism $T^0(\rho) \cong r_l(\pi)^{(l)}$ holds. We now explain how the other direction of the theorem follows from this implication.

Vice versa, suppose then that $T^0(\rho_1) \cong r_l(\pi)^{(l)}$ for $\rho_1$ an irreducible, cuspidal, Galois-invariant, $l$-adic representation of $\GL{n}{E}$ and $\pi$ an irreducible, cuspidal, $l$-adic representation of $\GL{n}{F}$. By the classical theory of base change we know (see e.g. section 6, chapter 1 of \cite{AC}) that there exists an irreducible, cuspidal $l$-adic representation $\pi_1$ of $\GL{n}{F}$ with $\rho_1 = \mathrm{bc}(\pi_1)$.

By compatibility of base change with reduction $\mod l$ (which follows immediately from the main theorem of \cite{BH} since restriction on the Galois side preserves congruence $\mod l$), we have that $r_l(\rho_1) \cong \mathrm{bc} \left( r_l(\pi_1) \right)$.
We obtain thus \[ r_l(\pi_1)^{(l)} \cong T^0(\rho_1) \cong r_l(\pi)^{(l)} \] where the left-hand side equality is the direction of the theorem that we are assuming right now, while the right-hand side equality holds by hypothesis.
Hence we get \[ r_l(\rho_1) \cong \mathrm{bc} \left( r_l(\pi_1) \right) \cong \mathrm{bc} \left( r_l(\pi) \right) \] as we needed.
\end{rem}
Finally, we want to remark that the theorem makes no mention of the ramification of the cyclic extension, but in fact the ramified case and the unramified one are intrinsically different and we prove the statement separately in each case.
\subsection{Overview of the paper}
Before giving an overview of the proof of Theorem \ref{mainthm} we want to emphasize the importance of the Shintani correspondence for our purposes. Shintani's seminal work was fundamental for the development of cyclic base change in the local setting, as described by Arthur and Clozel in \cite{AC}, chapter 1, section 6.

In fact, to prove theorem \ref{mainthm} in the unramified case one fundamental step is to rephrase the Shintani correspondence (as in \cite{Sh}), which gives a bijection between irreducible $l$-adic representations $\tilde \pi$ of $\GL{n} {\F{q}}$ and irreducible $l$-adic representations $\tilde \rho$ of $\GL{n} {\F{q^l}}$ which are Galois-fixed, in terms of a matching of modular representations (see Theorem \ref{charscorresp}). The bijection that Shintani discovered is given in terms of characters, making use of a generalization of the norm map - we recall it in the Notation section.

We show that in the cuspidal case (when we view both characters of $\tilde \pi$ and $\tilde \rho$ as valued in the integral closure of $\Zl$), the Shintani correspondence may be realized as in our main theorem: taking Tate cohomology of any stable $\overline \Zl$-lattice in the underlying space of $\tilde \rho$ with respect to the natural $\Frob_q$-action on $\GL{n}{ \F{q^l}}$ yields a $\mod l$ representation of $\GL{n} {\F{q}}$, and this coincides (up to a Frobenius twist) with the reduction $\mod l$ of $\tilde \pi$. This is proved explicitly, by computing characters of $\tilde \pi$ and $\tilde \rho$ and showing that the Brauer characters of the two modular representations match.

We then lift $\tilde \pi$ and $\tilde \rho$ to cuspidal, level zero $l$-adic representations of $\GL{n}{E}$ and $\GL{n}{F}$, and we show that the two lifts $\pi$ and $\rho$ match in the same way.

Finally, we show that the lifts $\pi$ and $\rho$ are in fact realizing cyclic base change: this is done by applying the local Langlands correspondence (described explicitly by Bushnell and Henniart in \cite{BH1} and \cite{BH2}) and proving that on the Galois side the representations one obtains are restriction of one another. \\

The ramified case requires different techniques (one cannot hope to work on the finite group level, as now the two local fields have the same residue field): here we start with a cuspidal $l$-adic representations $\pi$ of $\GL{n}{E}$, realize its base change $\rho$ via the explicit local Langlands correspondence (as in \cite{BH1} and \cite{BH2}) and only then we show that the Frobenius twist of the reduction of $\pi$ is isomorphic to the Tate cohomology of $\rho$.

This is done by considering both representations as $l$-sheaves on the respective coset spaces, which are then embedded as vertex sets of the respective Bruhat-Tits buildings. A careful analysis of Tate cohomology lets us conclude that $T^0(\rho)$ is in fact supported on the (canonically embedded) building of $\GL{n}{E}$ and eventually we can see that the two $l$-sheaves are isomorphic. \\

As explained above, we treated the simplest possible case (level zero and minimal-maximal type), but we expect a similar statement to be true in a much more general situation, without any restriction on the level and type.
\begin{conj} As before, let $F \supset E$ be a degree $l$ extension of characteristic zero, non-archimedean local fields of residue characteristic $p$, where $p$ and $l$ are different primes. Let $\pi$ be a cuspidal $l$-adic representation of $\GL{n}{E}$ and $\rho$ be a cuspidal $l$-adic representation of $\GL{n}{F}$ that is $\Gal(F/E)$-equivariant; extend $\rho$ to a representation of $\GL{n}{F} \rtimes \Gal(F/E)$. \\
Then $r_l(\pi)$ and $r_l(\rho)$ are in base change exactly when \[ r_l(\pi)^{(l)} \cong T^0(\rho). \]
\end{conj}
This conjecture is consistent with the general philosophy about linkage being compatible with the Langlands functorial transfer (Conjecture 6.3 in \cite{TV}).
\begin{rem} The conjecture as stated only deals with the case of two cuspidal representations being in base change. Nonetheless, when $l|n$ it is possible in certain cases to have $\pi$ cuspidal and its base change $\rho = \mathbf{bc}(\pi)$ a principal series. Some computations we carried seem to suggest that in this case $\pi$ is still linked to $\rho$ (that is, $r_l(\pi)^{(l)}$ is a Jordan-Holder factor of $T^*(\rho)$), but we do not have equality since $T^0(\rho)$ is now \emph{much bigger} than $r_l(\pi)^{(l)}$.
\end{rem}

The referee raised the interesting question about what happens if $F \supset E$ is an extension of positive characteristic, non-archimedean local fields - everything else being the same. We were primarly interested in the situation over $p$-adic fields (for instance, because the setup of \cite{TV} is for reductive groups over number fields) so we did not consider that case, but a brief glance at the main ingredients in the proof of the main theorem (several results of \cite{V}, the essential tame local Langlands correspondence of \cite{BH1} and \cite{BH2}, and the $l$-modular local Langlands correspondence as in \cite{BH} and \cite{Vig}) shows that each of them holds in the greater generality of $\ch E = \ch F = p$. Therefore, we expect the results of this paper to be easily extendable to the local function field setting.

\subsection{Acknowledgments} This paper would not have seen the light of day without the guidance of my advisor Akshay Venkatesh, who suggested the problem and shared his insights with me. I would also like to thank Daniel Litt for the numerous discussions concerning Local Langlands and Tate cohomology, Daniel Bump for a careful reading of a first draft and his many good advices, as well as Laurent Clozel for proposing improvements and new directions for this project. Finally, I acknowledge an anonymous referee for his/her useful comments.

\section{Notation}
In the course of the paper $k$ will denote a finite field and $k_m$ its unique extension of degree $m$. We use a similar notation for local fields: if $F$ is a finite extension of $\Qp$, $F_m$ indicates its unique unramified extension of degree $m$. \\
In both cases, the Galois groups of the relative extensions $k_m/k$ and $F_m/F$ are cyclic of order $m$, and we call Frobenius and denote $\Frob$ (sometimes also $\Frob_k$ or $\Frob_F$ to clarify what the ground field is) the arithmetic Frobenius, i.e. the automorphism $x \mapsto x^{|k|}$ in the finite field case or a lift of it in the unramified case. In both cases, $\Frob$ is a generator of the Galois group.

There is then an obvious entry-wise action of $\Frob$ on $\GL{n}{k_n}$ or $\GL{n}{F_n}$, which gives us a Galois action on respectively $\GL{n}{k_m}$ or $\GL{n}{F_m}$.

We will use boldface letters (like $\mathbf G$) for algebraic groups and usual font (e.g. $G$) for the group of their points over some field, which will be specified or clear in the context.

The notation $\lambda \vdash n$ indicates that $\lambda$ is a partition of the positive integer $n$.

Throughout this paper we will mainly deal with $l$-adic and $l$-modular representations and we will specify the coefficient field unless it is clear from the context. By an integral $l$-adic representation of a group $G$ we mean an $l$-adic representation $\pi$ over a $\Ql$ or $\overline \Ql$-vector space $V$ for which there exists a $G$-invariant lattice $L$.

We fix an isomorphism $\overline \Ql \cong \C$, which we will use to identify the complex and the $l$-adic local Langlands correspondences. As explained in \cite{Vig}, the bijection given by the $l$-adic local Langlands correspondence between representation of $W(L)$ and admissibile representations of $\GL{n}{L}$ (in particular, the bijection of lemma \ref{LLC} later) only depends on the choice of this isomorphism if the order of the residue field $k_L$ is not a square, and in that case it only depends on the choice of an element $\sqrt { |k_L| } \in \overline \Ql$.

\begin{defn}[Compact induction] \label{compind} Let $G$ be a locally compact, totally disconnected group and $H \subset G$ a closed subgroup. Let $(\sigma, W)$ be a representation of $H$. The induced representation $\Ind_H^G \sigma$ has underlying space \[ \Ind_H^G W = \left\{ f:G \lra W \, | \, f(hg) = \sigma(h) f(g) \, \forall h \in H, g \in G \textnormal{ and $f$ is locally constant} \right\}. \] The locally constant condition is tantamount to saying that the $G$-action on the right: \[ (g.f)(x) = f(xg) \qquad \forall g, x \in G \] gives rise to a smooth representation of $G$ on $\Ind_H^G W$.

The compactly induced representation $\cInd_H^G W$ is given by the restriction of the $G$-action above to the subspace of functions with compact support modulo $H$.
\end{defn}

We now recall the notions of cuspidal and supercuspidal representation with an ad-hoc definition for $\Gl{n}$.
\begin{defn}[Cuspidal and supercuspidal representations]
Let $G = \GL{n}{F}$ or $G= \GL{n}{k}$. \\
For any $ 1 \le  m \le n-1$, consider the maximal upper parabolic subgroup $P_m \subset G$ of matrices whose lower-left $(n-m) \times m$ block is zero. Let $U_m$ be its unipotent radical (consisting of matrices in $P_m$ which coincides with the identity outside of the upper-right $m \times (n-m)$ block). \\
Then the quotient map $ P_m \twoheadrightarrow P_m / U_m $ admits a natural section whose image (in $P_m$) consists of a Levi factor $M_m \cong \Gl{m} \times \Gl{n-m}$: the subgroup of $P_m$ of block-diagonal matrices with blocks of sizes $\{ m, n-m \}$.

Given a representation $\tau$ of $M_m \cong P_m / U_m$, one can inflate it trivially to $P_m$ via the quotient map, and then induce up to $G$: we denote the result of this process by $i^G_{P_m} \tau$.

Let $\pi$ be a representation of $G$. We say that $\pi$ is cuspidal if for all $1 \le m \le n-1$, and for all smooth, admissible representations $\tau$ of a Levi factor $M_m$, we have \[ \Hom_G \left( \pi, i^G_{P_m} \tau \right) = 0. \]
We say that $\pi$ is supercuspidal if for all $1 \le m \le n-1$, and for all representations $\tau$ of a Levi factor $M_m$, we have that $\pi$ is not a subquotient of $i^G_{P_m} \tau$.
\end{defn}
\begin{rem} Recall that in the $p$-adic case (i.e. when $G = \GL{n}{F}$) by representation we mean smooth representation. Notice moreover that in this case the quotient $G / P_m$ is compact, hence ``inducing from $P_m$ to $G$" is un-ambiguous since induction and compact induction coincide.
\end{rem}
As a warning to the reader we want to remark that, unlike in the classical case of irreducible, admissible, complex representations, for $\mod l$ representations the notion of supercuspidal (see for example \cite{V}, chapter 2, sections 2.2 to 2.5) is strictly stronger than that of cuspidal. We will always work with the latter.
\begin{defn}[Cyclic, Local Base Change - $l$-adic case] Let $\pi$ be an $l$-adic smooth irreducible cuspidal representation of $\GL{n}{E}$ and $\rho$ be an $l$-adic smooth irreducible cuspidal representation of $\GL{n}{F}$ which is Galois-invariant. Denote by $\mathcal L_K$ the bijection between the cuspidal, irreducible, $l$-adic  representations of $\GL{n}{K}$ and the $n$-dimensional, irreducible, $l$-adic representations of the Weil group $W(K)$ given by the Local Langlands Correspondence.

We say that $\pi$ and $\rho$ are in base change, and write $\rho = \mathrm{bc}_{\overline \Ql}(\pi)$ if we have \[ \mathrm{res}^{W(E)}_{W(F)} \left( \mathcal L_E (\pi) \right) \cong \mathcal L_F (\rho) \] as $l$-adic representations of $W(F)$.
\end{defn}
The notion of base change for complex representations of $\GL{n}{F}$ (and hence $l$-adic representations) is thus defined as the operation on the automorphic side corresponding to restriction on the Galois side via the Local Langlands correspondance. This prompt the following definition in the $l$-modular setting. 
\begin{defn}[Cyclic, Local Base Change - $l$-modular case] Let $\pi$ be a $\mod l$ smooth irreducible cuspidal representation of $\GL{n}{E}$ and $\rho$ be a $\mod l$ smooth irreducible cuspidal representation of $\GL{n}{F}$ which is Galois-invariant. We say that $\pi$ and $\rho$ are in $l$-modular base change, and write $\rho = \mathrm{bc}_{\overline \Fl}(\pi)$, if they admit lifts $\widetilde \pi$ and $\widetilde \rho$ to $l$-adic, cuspidal, irreducible representations (respectively of $\GL{n}{E}$ and $\GL{n}{F}$) such that \[ r_l \left( \res^{W(E)}_{W(F)} \left( \mathcal L_E (\widetilde \pi) \right) \right) \cong r_l \left( \mathcal L_F ( \widetilde \rho) \right). \]
\end{defn}
\begin{rem}
By the main theorem of \cite{BH}, the $\mod l$ reduction does not depend on the choice of the lift $\widetilde \pi$: indeed if $\widetilde \pi'$ is another such lift, the theorem says that the $l$-adic Galois representations $\mathcal L_E(\widetilde \pi)$ and $\mathcal L_E(\widetilde \pi')$ have isomorphic reduction $\mod l$, and hence so do their restrictions to $W(F)$. Similarly for the choice of the lift $\widetilde \rho$. \\
Given $\pi$ as above, if an $l$-modular base change $\rho$ exists, then it is unique. Indeed, again by the main theorem of \cite{BH}, the condition of having isomorphic $\mod l$ reduction $r_l \left( \res^{W(E)}_{W(F)} \left( \mathcal L_E (\widetilde \pi) \right) \right) \cong r_l \left( \mathcal L_F ( \widetilde \rho) \right) $ is equivalent to the identical condition on the automorphic side under $l$-adic local Langlands: $\widetilde \rho$ and $\mathcal L_E^{-1} \left( \res^{W(E)}_{W(F)} \left( \mathcal L_F (\widetilde \pi) \right) \right)$ have the same reduction $\mod l$. But $\rho = r_l(\widetilde \rho)$ by assumption, while the reduction of $\mathcal L_E^{-1} \left( \res^{W(E)}_{W(F)} \left( \mathcal L_F (\widetilde \pi) \right) \right)$ is independent of $\rho$.
\end{rem}
Notice that by the results in section III.5.10(2) of \cite{V} every irreducible, $\mod l$ cuspidal representation of $\Gl{n}$ admits a lift to an $l$-adic cuspidal representation, which is obviously irreducible - hence each pair $(\pi, \rho)$ as above admits $l$-adic lifts $(\widetilde \pi, \widetilde \rho)$ which one can apply the definition to.

The above definition is moreover compatible with the semisimple $\mod l$ Langlands correspondence as established by Vigneras, in the following sense:
\begin{defn}[Supercuspidal support, see \cite{Vig}, section 1.2] Let $R = \overline \Ql$ or $R = \overline \Fl$. Let $\pi$ be an irreducible representations of $\GL{n}{F}$ with coefficients in $R$. The supercuspidal support of $\pi$ is the formal sum $\mathrm{sc}(\pi) = \pi_1 + \ldots + \pi_h$ where $\pi_i$ is a supercuspidal, irreducible $R$-representation of $\GL{n_i}{F}$ (for $\sum_i n_i = n$) such that $\pi$ is a subquotient of the normalized parabolic induction $\pi_1 \times \ldots \times \pi_h$. This is well-defined, as explained in the remark of section 1.2 of \cite{Vig}.
\end{defn}
We remark that, as explained in section 1.5 of \cite{Vig}, the supercuspidal support behaves well under reduction $\mod l$: if $\pi$ is an irreducible integral $l$-adic representation of $\GL{n}{F}$ with supercuspidal support $\mathrm{sc}(\pi) = \pi_1 + \ldots + \pi_h$, then each irreducible subquotient of the reduction $r_l(\pi)$ has ($\mod l$) supercuspidal support equal to $\mathrm{sc} \left( r_l(\pi_1) \right) +\ldots +  \mathrm{sc} \left(  r_l(\pi_h) \right)$. We define this formal sum to be the supercuspidal support $\mathrm{sc}(r_l(\pi))$ of the reduction $\mod l$ of $\pi$.

The $l$-modular semisimple local Langlands correspondence is then encoded in the following important result.
\begin{thm}[Vigneras, thm. 1.6 in \cite{Vig}] \label{modularlocallanglands}
Let $\pi$ and $\pi'$ be irreducible integral $l$-adic representation of $\GL{n}{F}$ and $\sigma$, $\sigma'$ be $n$-dimensional integral $l$-adic representation of the Weil group $W(F)$ in semisimple Langlands correspondence, i.e. $\mathrm{sc}(\pi) \leftrightarrow \sigma$ and $\mathrm{sc}(\pi') \leftrightarrow \sigma'$. Then \begin{enumerate}
\item The supercuspidal supports of the reductions coincide  - i.e. $\mathrm{sc}( r_l(\pi)) = \mathrm{sc} (r_l(\pi'))$ - if and only the reductions of the Weil representations are isomorphic, that is, $r_l(\sigma) \cong r_l(\sigma')$.
\item There exists a unique compatible sistem of bijections indexed by the natural numbers $n \ge 1$: \[ \mathrm{Supercusp}_{\overline \Fl} (\GL{n}{F}) \longleftrightarrow \Irrep_{\overline \Fl} ( W(F) )[n] \] between supercuspidal $\mod l$ representations of $\GL{n}{F}$ and irreducible $n$-dimensional $\mod l$ representations of $W(F)$ which induces via the supercuspidal support a semisimple Langlands correspondence over $\overline \Fl$: \[ \pi \to \mathrm{sc}(\pi) \leftrightarrow \sigma \quad \Irrep_{ \overline \Fl} (\GL{n}{F}) \lra \Rep_{\overline \Fl} (W(F)) \] compatible with $\mod l$ reduction, i.e. given $\pi \in \Irrep_{\overline \Ql} (\GL{n}{F})$ and an $n$-dimensional $\sigma \in \Rep_{\overline \Ql} (W(F))$ such that $\mathrm{sc}(\pi)$ is in $l$-adic Langlands correspondence with $\sigma$, then $\mathrm{sc}(r_l (\pi))$ is in $\mod l$ Langlands correspondence with $r_l(\sigma)$.
\end{enumerate}
\end{thm}
Suppose now that $\pi$ (respectively $\rho$) is an irreducible, supercuspidal, $\mod l$ representations of $\GL{n}{E}$ (respectively, $\GL{n}{F}$). According to our definition, these are in base change if for two cuspidal $l$-adic lifts $\widetilde \pi$ and $\widetilde \rho$ we have $r_l \left( \mathcal L_E( \widetilde \pi) \right)|_{W(F)} \cong r_l \left( \mathcal L_F(\widetilde \rho) \right)$. \\
On the other hand, Vigneras' theorem provides us with the $l$-modular representations $^{G} \pi \in \Irrep_{\overline \Fl}( W(F))$ and $^{G} \rho \in \Irrep_{\overline \Fl}( W(E))$, and compatibility with reduction $\mod l$ means exactly that $^{G} \pi = r_l \left( \mathcal L_E( \widetilde \pi) \right)$ and $^{G} \rho = r_l \left( \mathcal L_F( \widetilde \rho) \right)$.
Therefore, for supercuspidal $\pi$ and $\rho$, $l$-modular base change corresponds exactly to restriction on the Galois side under local Langlands, without the necessity of taking $l$-adic lifts. \\

We recall here the definition of Shintani s correspondence, since it will be used later. Let $k_m / k$ be an extension of finite fields with $\gamma$ the Frobenius element and $(\rho, V)$ be a Galois-invariant, irreducible $l$-adic representation of $\GL{n}{k_m}$, so that there exists an operator $T_{\gamma}: (\rho \circ \gamma, V) \lra (\rho, V)$ realizing the isomorphism of representations.
\begin{thm}[Theorem 1 in \cite{Sh}] For a suitable normalization of $T_{\gamma}$, there exists an irreducible character $\chi_{\rho}$ of $\GL{n}{k}$ which satisfies \[ \Tr \left( I_{\gamma} (\rho(g)) \right) = \chi_{\rho} \left( \Norm{k_m}{k} (g) \right) \qquad \forall g \in \GL{n}{k_m} \] where the norm map is defined as $\Norm{k_m}{k}(g) = \gamma^{m-1}(g) \cdot \ldots \cdot \gamma(g) \cdot g$. \\
Moreover, the mapping $\rho \mapsto \chi_{\rho}$ is a bijection between the Galois-invariant irreducible representations of $\GL{n}{k_m}$ and the irreducible representations of $\GL{n}{k}$.
\end{thm}
The Shintani base change is thus the inverse of the bijection above, so that $\rho$ is the Shintani base change of the representation having character $\chi_{\rho}$. \\

Consider a modular representation $(\alpha, V)$ of a group $G$ (that is, a representation defined over the $\overline \Fl$-vector space $V$). We denote by $\alpha^{(l)}$ its Frobenius twist, that is the representation where the map $\alpha: G \lra \mathrm{GL}(V)$ is the same, as well as the underlying set of $V$, but the scalar action of $\overline \Fl$ on $V$ is twisted by Frobenius: \[ V^{(l)}:=V \otimes_{(\overline \Fl, \Frob_l)} \overline \Fl \] so that scalar multiplication $*$ on $V^{(l)}$ is given by \begin{equation} \label{deftwist} \lambda * v = \lambda^{\frac{1}{l}} \cdot v \qquad \forall \lambda \in \overline \Fl, \, v \in V. \footnote{We denote by $\lambda^{1/l}$ the \emph{unique} $l$-th root of $\lambda \in \overline \Fl$, which is a well-posed notion since the Frobenius automorphism $x \mapsto x^l$ of $\overline \Fl$ is a bijection.} \end{equation}

For the notion of type of a level zero representation see Vigneras in \cite{V}, chapter III, section 3.1; in particular a representation of level zero having minimal-maximal type is one whose restriction to a maximal parahoric $K^0$ (that is, a maximal compact in $\GL{n}{F}$, necessarily conjugate to $\GL{n}{\OO_F}$) contains a representation $\sigma$ of $K^0$ trivial on the first congruence subgroup $K^1$; $\sigma$ hence corresponds to the inflation (under the quotient map $K^0 \twoheadrightarrow K^0 / K^1 \cong \GL{n}{k_F}$) of a cuspidal representation of $\GL{n}{k_F}$. An identical description applies with $E$ in the place of $F$.

We say that a representation $\rho$ of $\GL{n}{F}$ is $\Gal(F/E)$-equivariant when for each element $\gamma \in \Gal(F/E)$ the representations $\rho$ and $\rho \circ \gamma$ are abstractly isomorphic (here $\gamma$ acts componentwise on $\GL{n}{F}$).
In the situation of Theorem \ref{mainthm}, once one fixes a generator $\gamma$ for the cyclic $l$-group $\Gal(F/E)$, the condition of being $\Gal(F/E)$-equivariant is equivalent to $\rho \cong \rho \circ \gamma$ for that particular generator: one can then extend $\rho$ \emph{non-uniquely} to a representation of $\GL{n}{F} \rtimes \Gal(F/E)$. \footnote{In details, choose an intertwining operator $T: \left( \rho , V \right) \lra \left( \rho  \circ \gamma, V \right)$ realizing the isomorphism, and then extend $\rho$ to $\Rho: \GL{n}{F} \rtimes \Gal(F/E) \lra \mathrm{GL}(V)$ as $\Rho(g, \gamma^k) = \rho(g) \circ T^k \in \mathrm{GL}(V)$. The group operation on the semidirect product is given as $(h_1, \gamma^{k_1})(h_2, \gamma^{k_2}) = (h_1 \gamma^{k_1}(h_2), \gamma^{k_1+k_2})$.}

Notice that both representations in the theorem, and in fact all $l$-adic representations that will be mentioned in this paper, are $l$-integral: that is, there exists a $\overline \Zl$-lattice $L \subset V$ (where $V$ is the underlying space of the $l$-adic representation) which is stable under the group action and generates $V$ (over $\overline \Ql$). For a proof that such a lattice exists, and that in fact it can be realized over a finite extension of $\Ql$, see \cite{V}, chapter II, section 4.12. We denote by $\overline \Zl$ the integral closure of $\Zl$ in a fixed separable closure $\overline \Ql$ of $\Ql$.

Suppose we have then a $\Gal(F/E)$-equivariant $l$-adic representation $\left( \rho , V \right)$ of $\GL{n}{F}$ which is $l$-integral, and fix a generator $\gamma \in \Gal(F/E)$. To define Tate cohomology we need to fix an isomorphism of $\GL{n}{F}$-representations \begin{equation} \label{defnT} \mathrm T: \left( \rho, V \right)  \stackrel{\cong}{\lra} \left( \rho \circ \gamma, V \right); \end{equation} and then \emph{pick} a $\GL{n}{F}$-invariant lattice $L \subset V$ which is preserved by $\mathrm T$. 

We denote by $N = \id + \mathrm T + \ldots + \mathrm T^{l-1}$ the ``norm" and we have the obvious $2$-periodic chain complex of $\overline \Zl$-modules \[ \ldots L \stackrel{\id - \mathrm T}{\lra} L \stackrel{N}{\lra} L \stackrel{\id - \mathrm T}{\lra} L \stackrel{N}{\lra} L \ldots \]
The cohomology of this complex is by definition Tate cohomology: \[ T^0(L) = \ker (\id - \mathrm T) / \im(N) \qquad T^1(L) = \ker(N) / \im(\id -\mathrm T); \] these are naturally $\overline \Fl$-modules with an action of the groups of the fixed points $\GL{n}{F}^{\gamma} = \GL{n}{E}$: hence one gets two $l$-modular representations of $\GL{n}{E}$.
We will denote their semisimplifications by $T^i(\rho)$; in general, these may depend on the choices of $\mathrm T$ and of an invariant lattice $L$ but the following lemma will apply to all the $l$-adic representations $\rho$ considered in this paper.
\begin{lem} \label{T0welldefined} Suppose $\rho$ as above is irreducible and cuspidal. Then $T^0(\rho)$ is well-defined and independent of the choices of $\mathrm T$ and $L$.
\end{lem}
\begin{proof}
Notice that by Schur's lemma, $\mathrm T^l$ is a scalar $\lambda \in \overline \Ql$, hence by replacing $\mathrm T$ with $\lambda^{1/l} \mathrm T$, we can assume that $\mathrm T$ had finite order $l$: the only ambiguity left in the definition of $\mathrm T$ is the choice of $l$-th root of $\lambda$, i.e. $\mathrm T$ is well-defined up to an $l$-th root of unity in $\overline \Ql$. \\
Fix then one such operator $\mathrm T$ with $\mathrm T^l = \id$. By proposition 4.5 in \cite{BH} (and the discussion preceding it in sections 4.3 and 4.4), the irreducible and cuspidal representation $\rho$ admits a unique stable lattice $L$ up to homothety. Since $\mathrm T(L)$ is another stable lattice, we must have $\mathrm T(L) = z L$ for some $z \in \overline \Ql^*$: then $L  =\id(L) = \mathrm T^l(L) = z^l L$ shows that $z \in \overline \Zl^*$, and in particular $\mathrm T(L) = L$ - i.e. the intertwining operator preserves any stable lattice $L$. Fix one such stable lattice $L$: we now prove that $T^0(L)$ is independent of this choice. \\
Consider the short exact sequence of $\GL{n}{F} \rtimes \Gal(F/E)$-modules induced by ``multiplication by $l$" on $L$: \[ 0 \lra L \stackrel{\cdot l}{\lra} L \lra L / lL \lra 0 \] and take Tate cohomology with respect to the fixed choice of $\mathrm T$.
The long exact sequence in cohomology reads \[ \ldots \lra T^0(L) \stackrel{\cdot l}{\lra} T^0(L) \lra T^0 \left( L / l L \right) \lra T^1(L) \lra \ldots \]
Since $T^0(L)$ is an $\overline \Fl$-module, multiplication by $l$ factors through the zero map, so our long exact sequence becomes \[ 0 \lra T^0(L) \lra T^0 \left( L / l L \right) \lra T^1(L) \lra \ldots \]
It remains to prove that $T^1(L) = 0$, let us assume this for the time being to conclude the proof of the lemma. $T^1(L) = 0$ yields that \[ T^0(L) \cong T^0 \left( L / l L \right), \] so that $T^0(L)$ only depends on the reduction $L / l L$, which is an irreducible $\mod l$-representation by theorem III.1.1(d) in \cite{V}. In particular, $L / l L$ is the irreducible $\mod l$ reduction of $\rho$, which is independent of the choice of lattice $L$ by the Brauer-Nesbitt principle. Hence, $T^0(L)$ will be independent of the choice of the lattice $L$. \\
Finally, we prove independence on the choice of the intertwiner $\mathrm T$ (such that $\mathrm T^l = \id$): let then $\mathrm T$ and $\mathrm T' = \zeta_l \mathrm T$ be two such intertwiners, where $\zeta_l$ is an $l$-th root of unity: notice that a $\GL{n}{F}$-stable lattice $L$ will be preserved by $\mathrm T$ if and only if is preserved by $\mathrm T'$, and we have $T^0_{\mathrm T}(L) \cong T^0_{\mathrm T} (L / lL)$ and similarly $T^0_{\mathrm T'}(L) \cong T^0_{\mathrm T'} (L / lL)$. Now proposition 6.1 in \cite{TV} proves that in the modular setting the Tate cohomology is, up to isomorphism, independent of the choice of the intertwiner $\mathrm T$, hence $T^0_{\mathrm T} (L / lL) \cong T^0_{\mathrm T'}(L / l L)$ which concludes the proof of the lemma together with the following theorem.
\end{proof}
\begin{thm} \label{T1dies} Let $\rho$ be an irreducible, integral, cuspidal $l$-adic representation of $\GL{n}{F}$ (resp. $\GL{n}{k_l}$) which is $\Gal(F/E)$-invariant (resp. $\Gal(k_l/k)$-invariant). Then $T^1(\rho) = 0$.
\end{thm}
\begin{proof} We will give the proof in the case of $\GL{n}{F}$: the proof in the finite field setting is exactly identical, since the main ingredients are the Kirillov model - which exists both for $\GL{n}{F}$ and for $\GL{n}{k_l}$ - and an explicit computation, which is identical in the two cases.

Consider then the Kirillov model: denoting by $\psi: F \lra \overline \Ql^*$ a non-degenerate character of the upper unipotent subgroup $U_n$, we can consider the mirabolic representation $\tau = \cInd^{P_n}_{U_n} \psi$, where $P_n$ is the mirabolic subgroup. By the discussion in section III.1.1 of \cite{V}, $\tau$ is irreducible, $l$-integral and with irreducible reduction $\mod l$. Moreover, we have an isomorphism $\res^{\GL{n}{F}}_{P_n} \rho \cong \tau$ as representations of $P_n$.

Therefore it suffices to show that $T^1\left( \rho|_{P_n(F)} \right) = 0$ since $P_n(F)^{\Gal(F/E)} = P_n(E)$ and restriction of the group action to a subgroup is obviously compatible with the procedure described to define $T^1$. \\
Consider then $\tau \cong \rho|_{P_n(F)}$: by assumption, we have $\tau \cong \tau \circ \gamma$ for a generator $\gamma$ of $\Gal(F/E)$, and we can choose the intertwiner to be \[ \mathrm T: \tau \lra \tau \circ \gamma \textnormal{ defined as } \left( \mathrm T f \right)(g) = f \left( \gamma^{-1}g \right) \] where we use the function model for the compactly induced representation $\tau$, as in definition \ref{compind}.

Fix a choice of coset representatives $\left\{ U_n(F) g \right\}$ for $U_n(F) \backslash P_n(F)$ such that $U_n(F) \gamma(g)$ is another coset representative: in other words, $\Gal(F/E)$ acts on these cosets; this can be done since $\gamma$ preserves $U_n(F)$. \\
Pick then the $\overline \Zl$-lattice $L \subset \tau$ generated by the characteristic functions $1_g$, where $1_g$ is the function supported on the coset $U_n(F)g$ and mapping $g \mapsto 1$. Since we have $\mathrm T \left( 1_g \right)= 1_{ \gamma(g)}$, the lattice $L$ is preserved by the intertwiner operator. \\
Recall that $T^1(\rho) = \ker N / \im(1- \mathrm T)$. Then a function $f = \sum_g a_g 1_g$ is in the kernel of the norm operator $N$ exactly when \[ 0 = Nf = \sum_{g = \gamma(g)} a_g ( l \cdot 1_g ) + \sum_{g \neq \gamma(g)} a_g \left( \sum_{i=0}^{l-1} 1_{\gamma^i(g)} \right) = \sum_{g = \gamma(g)} (a_g \cdot l) 1_g + \sum_{g \neq \gamma(g)} \left( \sum_{i=0}^{l-1} a_{\gamma^i(g)} \right) 1_g. \]
Now the coefficients $a_g$ are in $\overline \Zl$ which is torsion-free, hence for each $\gamma$-fixed $g$ we must have $a_g \cdot l = 0$ which implies $a_g=0$. For each $g$ which is not Galois-fixed, we must instead have $\sum_{i=0}^{l-1} a_{\gamma^i(g)} = 0$.

A basis of $\ker (N)$ is then given by the functions $\left\{ 1_g - 1_{\gamma(g)} \right\}$ as $g$ varies among non-Galois fixed elements. But then it is immediate to see that $1_g - 1_{\gamma(g)} = (1- \mathrm T) 1_g \in \im (1- \mathrm T)$ which proves that $T^1(\rho) =0 $.
\end{proof}

\section{Characters of finite general linear groups}
In this section we describe the theory of characters for the finite general linear group $\GL{n}{k}$. This has been completely worked out by Green in \cite{G} in the complex case. The main idea is to describe conjugacy classes $c$ in $\GL{n}{\F{q}}$ by using representatives in the form of companion matrices, and parametrize those by partition-valued functions $\nu_c$ on the set of irreducible polynomials with $\F{q}$-coefficients. Some amount of combinatorics is then required to explicitly describe the characters, in particular the Green polynomials (see for example Macdonald in \cite{M}) are brought into the picture and to explicitly compute a given character on an element it is necessary to know the value of a specific Green polynomial at some power of $q$.

The description is much simpler for cuspidal characters: both the construction and the explicit computation are easier in this case, and the required Green polynomials have a very neat form for any general $n$. We follow the construction of Springer in $\cite{S2}$ to construct the cuspidal characters.

Let $\mathbf G$ be a reductive algebraic group over $k$. A maximal $k$-torus $\mathbf T \subset \mathbf G$ is called \emph{minisotropic} if its $k$-rank (the dimension of the maximal split $k$-subtorus) is as small as possible. For $\mathbf G = \mathbf {GL}_n$, this means that the maximal $k$-split subtorus is the center. In particular \[ \mathbf T = \Res {k_n}{k}{\mathbf G_m} \] is a minisotropic torus for the $k$-group $\mathbf {GL}_n$, and its $k$-points are canonically isomorphic to $k_n^*$ (up to the action of the cyclic Galois group of $k_n / k$).

In a similar way, when $\lambda=(n_1, \ldots, n_r)$ is a partition of $n$ we can define \[ F_{\lambda} = \prod_{ 1 \le i \le r} k_{n_i} \] which is a $k$-algebra where $k$ embeds diagonally into each factor. We then obtain a torus \[ \mathbf T_{\lambda} = \Res{F_{\lambda}}{k}{\mathbf G_m} \] that is embedded in $\mathbf {GL}_n$ by having it act on $\Res{F_{\lambda}}{k}{\mathbf A^1}$ - here $\mathbf A^1 = \mathrm{Spec} F_{\lambda}[x] $ is the affine line over $F_{\lambda}$. This embedding leads to an isomorphism \[ \mathbf T_{\lambda}(k) = T_{\lambda} \cong \prod_i k_{n_i}^*. \] In fact, both the embedding and the isomorphism are only defined up to conjugation by the Weyl group $\mathbf W_{\lambda}$.
\begin{exmp}
If the partition $\lambda$ has no repeated elements, the Weyl group $W_{\lambda} = \mathbf W_{\lambda} (k)$ is just the product of the cyclic Galois groups corresponding to each extension $k_{n_i} / k$. In general, $W_{\lambda}$ will also permute copies of Weil restrictions of the same degree.
\end{exmp}

\begin{rem} It is a general fact of the theory of reductive groups over finite fields that every torus of $\mathbf {GL}_n$ is realized as above, for some partition $\lambda$ and some choice of embedding $\mathbf T_{\lambda} \hookrightarrow \mathbf {GL}_n$.
\end{rem}

For a torus $T_{\lambda}$ consider now the space $K(T_{\lambda})$ of $\overline{\Ql}$-valued functions on it, and let $I \left( T_{\lambda} \right)$ be the subspace of $W_{\lambda}$-invariant functions. Denote \[ I_n = \bigoplus_{\lambda \vdash n} I \left( T_{\lambda} \right). \]
\begin{thm} \label{isom} Let $\Cl(\GL{n}{k})$ be the space of $\overline{\Ql}$-valued class functions on $\GL{n}{k}$. The function $\gamma_n: I_n \lra \Cl(\GL{n}{k})$ defined as follows is an isomorphism:
\begin{equation} \left( \gamma_n \Phi \right)(x) = \sum_{\lambda \vdash n} | W_{\lambda} |^{-1} \sum_{t \in T_{\lambda}} H(x ; \lambda, t) \Phi_{\lambda}(t) \qquad x \in \GL{n}{k}
\end{equation}
where $\Phi = \left( \Phi_{\lambda} \right) \in \bigoplus_{\lambda \vdash n} I \left( T_{\lambda} \right)$ and $H(x; \lambda, t)$ is defined in terms of the Green polynomials as in formula 5.18 of \cite{S2}.
\end{thm}
\begin{proof} This is theorem 6.3 of \cite{S2}.
\end{proof}
The inverse map is denoted \[ \pi_n: \Cl(\GL{n}{k}) \lra I_n \] having components $\pi_n = \left( \pi_{\lambda} \right)_{\lambda \vdash n}$. The functions $\pi_{\lambda} f \in I \left( T_{\lambda} \right)$ are called the \emph{principal parts} of $f$.

We recall here some facts about the Green polynomials which we will need later (see \cite{S2} for a proof).
\begin{prop}\label{greenpolys}
Let $(\lambda, \mu)$ be an ordered pair of partitions with the same length $| \lambda| = | \mu | \le n$. Then the Green polynomial $Q(\lambda, \mu)$ is an integer, and in fact it is a polynomial function in $q = | k |$ with integer coefficients. Hence we can see it as a polynomial map \[ Q(\lambda, \mu): \{ \textnormal{ finite fields } \} \lra \Z \] which evaluates the indeterminate at the size of the finite field chosen. \\
Moreover, if $ | \lambda | = n $ and $r$ is the number of parts of $\lambda$, then \[ Q(\{ n\} , \lambda) = 1, \quad Q(\lambda, \{ n \})= (1-T) \ldots (1-T^{r-1}), \quad Q(\{ 1^n \}, \lambda) = (-1)^r \frac{(1-T) \ldots (1-T^q)}{\prod_{i \le n} (T^i -1)^{r_i(\lambda)}} \] where $r_i(\lambda)$ is the number of parts of $\lambda$ which are at least $i$.
\end{prop}

Following the construction above, the choice $\lambda = (n)$ gives a torus $\mathbf T_n= \Res{k_n}{k}{\mathbf G_m}$ which is a minisotropic torus for $\mathbf {GL}_n$. Its character group $\widehat {T_n} = \Hom \left( T_n, \overline{\Zl}^* \right)$ is acted upon by the Weyl group $W_n$, and we call a character $\phi \in \widehat{T_n}$ \emph{regular} if its stabilizer under this action is trivial. \\
Up to conjugation, we have that $T_n \cong k_n^*$ and $W_n \cong \Gal(k_n / k)$ so that a regular character is one which has largest possible orbit under a generator of the cyclic group $W_n$.
\begin{thm}[Cuspidal characters] \label{cuspidalchar}
Let $\phi \in \widehat {T_n}$ be a regular character. Then there exists an irreducible $l$-adic cuspidal representation of $\GL{n}{k}$ whose character $\chi_n(\phi)$ is such that
\begin{enumerate}
\item its principal part are all zero besides the one for $\lambda=(n)$, for which \begin{equation} \left( \pi_{(n)} (\chi_n(\phi)) \right) (t) = (-1)^{n-1} \sum_{w \in W_n} \phi(w.t) \quad \forall t \in T_n. \end{equation}
\item $\deg \chi_n(\phi) = \prod_{i=1}^n \left( |k|^i -1 \right)$.
\item Let $\phi_1, \phi_2 \in \widehat{T_n}$ be two regular characters. Then $\chi_n(\phi_1) = \chi_n(\phi_2)$ if and only if $\phi_1 = \left( \phi_2 \right)^{|k|^i}$ for some $i$.
\item The characters $\chi_n(\phi)$ as $\phi$ varies among the regular elements of $\widehat {T_n}$ exhaust the cuspidal part of the spectrum (that is, every irreducible $l$-adic cuspidal representation has character one of the $\chi_n(\phi)$'s).
\end{enumerate}
\end{thm}
\begin{proof} This consists of theorems 7.12 and 8.6 of \cite{S2}.
\end{proof}

We can then apply the above construction for both $k$ and $k_l$ and obtain explicitly all the cuspidal $l$-adic representations for $\GL{n}{k}$ and $\GL{n}{k_l}$; the next step is figuring out which cuspidal $l$-adic representations of the larger group are fixed under the Frobenius element $\Frob_k \in \Gal(k_l / k)$.

\section{The correspondence via Tate cohomology}

The following fact answers the question posed at the end of the previous section.
\begin{prop} Let $(n,l)=1$. Then a cuspidal $l$-adic representation $\rho$ of $\GL{n}{k_l}$ having character $\chi_n(\phi)$ for some regular $\phi \in \widehat{T_n} \cong \Hom \left( k_{ln}^*,  \overline{\Zl}^* \right)$ is Frobenius-fixed (that is, $\rho \cong \rho \circ \Frob_k$ for $\Frob_k \in \Gal \left( k_l / k \right)$) if and only if $\phi = \psi \circ \Norm{k_{ln}}{k_n}$ for some regular $\psi \in \widehat{T'_n} \cong \Hom \left( k_n^*, \overline{\Zl}^* \right)$.
\end{prop}
Notice that we will then have a natural choice for the cuspidal $l$-adic representation of $\GL{n}{k}$ correspondent to $\rho$: the one with character $\chi_n(\psi)$.
\begin{proof} We apply theorem \ref{cuspidalchar} to $\GL{n}{k_l}$: the only nonzero principal part of $\chi_n(\phi)$ is the one correspondent to $\lambda=(n)$, so we can assume that $\chi_n(\phi)$ takes value on $T_n \subset \GL{n}{k_l}$, with $T_n \cong {k_{ln}}^*$. Then the action of $\Frob_k$ is given by \[ \left( \Frob_k . \chi_n(\phi) \right)(t) = \chi_n(\phi)(t^{|k|}) = \chi_n \left( \phi^{|k|} \right)(t) \] by using the explicit formula in theorem \ref{cuspidalchar}. Since $\left( \Frob_k . \chi_n(\phi) \right)$ is still clearly a cuspidal character, the third statement of the same theorem tells us that $\rho$ is Frobenius-fixed if and only if \begin{equation}\label{frobfixed} \phi^{|k|} = \phi^{{|k|}^{li}} \textnormal{ for some }i. \end{equation}
It remains to show that condition \ref{frobfixed} is equivalent to $\phi$ factoring through the norm $N = \Norm{k_{ln}}{k_n}$ and to a regular character $\psi$. Let then $\alpha \in k_{ln}^*$ be a generator, and let $t = \phi(\alpha) $. \\
Using Hilbert 90, $\phi$ factors through $N$ if and only if $t \in \mu_{|k_n|-1}( \overline \Zl)$, while condition \ref{frobfixed} is equivalent to $t \in \mu_{|k| \left( |k|^{(li-1)} -1 \right)}( \overline \Zl)$ for some $i$. Since clearly $t \in \mu_{|k_{ln}|-1}( \overline \Zl)$, condition \ref{frobfixed} is actually equivalent to $t \in \mu_{ |k|^{(li-1)} -1 }( \overline \Zl)$. It's then easy to see that the last condition is equivalent to the one having $\phi$ factoring through $N$, under the assumption that $(l,n)=1$.

Assume then that $\phi = \psi \circ \Norm{k_{ln}}{k_n}$, it remains to prove that $\phi$ is regular if and only if $\psi$ is. Now $\phi$ is regular $\iff$ $t \notin \mu_{q^{lk}-1}( \overline \Zl)$ for all $1 \le k \le n-1$, while $\psi$ is regular $\iff$ $t \notin \mu_{q^k -1}( \overline \Zl)$ for all $1 \le k \le n-1$. The ``only if" implication is obvious, the other implication follows from $(l,n)=1$.
\end{proof}

We are now ready to state an intermediate important result: the following theorem is roughly saying that Shintani correspondence for cuspidal $l$-adic representations is realized by taking the Frobenius twist of the reduction $\mod l$. As discussed in the introduction, the Shintani correspondence (see \cite{Sh}) is a natural bijection between irreducible representations of $\GL{n}{k}$ and irreducible representations of $\GL{n}{k_l}$ which are $\Frob_k$-fixed, and essentially encodes the content of local cyclic base change in the unramified setting.
\begin{thm}\label{charscorresp} Let $\phi = \psi \circ \Norm{k_{ln}}{k_n}$ be a regular character of $k_{ln}^*$ which induces a Frobenius-fixed cuspidal $l$-adic representation of $\GL{n}{k_l}$. Then the reduction $\mod l$ of $\chi_n(\phi)$ (a character of $\GL{n}{k_l}$) coincides, when restricted to $\GL{n}{k}$, with the Frobenius twist of the reduction $\mod l$ of $\chi_n(\psi)$, the character of the cuspidal $l$-adic representation of $\GL{n}{k}$ induced by $\psi$. Equivalently,
\begin{equation} \overline{ \chi_n(\phi)(x) } = \overline{\chi_n(\psi)(x)}^l \qquad \forall x \in \GL{n}{k}.
\end{equation}
\end{thm}
\begin{rem}
Implicit in the above equation is that both $\chi_n(\phi)(x) $ and $\chi_n(\psi)(x)$ are algebraic numbers that are $l$-integral (that is, lie in $\overline{\Zl}$, the ring of integers of the algebraic closure of $\Ql$), and then the bar means "reduction $\mod \mathfrak l$", where $\mathfrak l \subset \overline \Zl$ is the maximal ideal.
\end{rem}
\begin{proof} Fix $x \in \GL{n}{k}$ for which we will verify the formula. \\
Consider first $\chi_n(\phi)$: by theorem \ref{cuspidalchar} the only nonzero principal part of $\chi_n(\phi)$ is for $\lambda = \{ n \}$ so we have \[ \chi_n(\phi)(x) = \left( \gamma_n \circ \pi \right) (\chi_n(\phi)) (x) = \frac{1}{|W_n|} \sum_{t \in T_n} H_{k_l} \left( x; \{ n \} , t \right) \pi_n \left( \chi_n(\phi) \right) (t) = \] \[ = \frac{(-1)^{n-1}}{|W_n|} \sum_{t \in T_n} H_{k_l} \left( x; \{ n \} , t \right) \sum_{w \in W_n} \phi(w.t) \] where the subscript in the function $H$ remembers what field we are evaluating the corresponding Green polynomials on, and we are also fixing isomorphisms $T_n \cong k_{ln}^*$ and $W_n \cong \Gal(k_{nl} / k_l) = \langle w_0 \rangle$ where $w_0.t = t^{q^l}$ has order $n$. \\
Similarly, \[ \chi_n(\psi)(x) = \left( \gamma_n \circ \pi \right) (\chi_n(\psi)) (x) = \frac{1}{|W'_n|} \sum_{t \in T'_n} H_k \left( x; \{ n \} , t \right) \pi_n \left( \chi_n(\psi) \right) (t) = \] \[ = \frac{(-1)^{n-1}}{|W'_n|} \sum_{t \in T'_n} H_k \left( x; \{ n \} , t \right) \sum_{w \in W'_n} \psi(w.t) \] where $T'_n \cong k_n^*$ and $W'_n \cong \Gal(k_n / k) = \langle w'_0 \rangle$ with $w'_0.t = t^q$ having order $n$. 

Now lemma 5.19 in \cite{S2} tells us that
\begin{equation} H(x; \lambda, t) = \left\{ \begin{array}{cc} 0 & \textnormal{if the semisimple part of $x$ is not conjugate to } t \in T_{\lambda} \subset \GL{n}{k} \\ 1 & \textnormal{if $x$ is a regular element whose semisimple part is conjugate to } t \in T_{\lambda} \subset \GL{n}{k}.
\end{array} \right. \end{equation}
Since $x \in \GL{n}{k}$, also $x^{ss} \in \GL{n}{k}$ and the latter is $\GL{n}{k}$-conjugate to the companion matrix of the characteristic polynomial $p(T) \in k[T]$ of $x^{ss}$. If $t \in T_n \cong k_{ln}^*$ is $\GL{n}{k_l}$-conjugate to $x^{ss}$, then $t$ is in fact a root of $p(T)$ and hence belongs to $k_n^*$. This shows that the set of $t \in T_n$ such that $t$ is $\GL{n}{k_l}$-conjugate to $ x^{ss}$ are exactly the $t' \in T'_n$ such that $t'$ is $\GL{n}{k}$-conjugate to $ x^{ss}$. Since taking $l$-power is an automorphism of $\FFl$, it suffices to show that for any such fixed $t$ \[ \overline{ H_{k_l} \left( x; \{ n \} , t \right) \sum_{w \in W_n} \phi(w.t) } = \overline { H_k \left( x; \{ n \} , t \right) \sum_{w \in W'_n} \psi(w.t) }^l. \]
We consider first the function $H$. Consider the explicit formula 5.18 in \cite{S2}: \[ H(x; \{ n \}, t) = \prod_{f \in F} Q \big( \nu_x(f); \rho(\{ n \}, t; f) \big) (|k|^ {\deg f}) \] where $F$ is the set of monic irreducible polynomials over $k$ besides $f(T)=T$, and $Q\left( -, - \right)$ is a Green polynomial depending on a pair of partitions, as in proposition \ref{greenpolys}. \\
In order to unravel the last expression, we recall the definitions of the two partitions that we are computing the Green polynomials at. The partition $\nu_x(f)$ is defined in section 2 of \cite{S2} (see in particular lemma 2.3 and the preceding discussion): letting $x = x^{ss} x^u$ be the Jordan decomposition of $x$, the conjugacy class of $x$ is uniquely determined by the data of the semisimple conjugacy class of $x^{ss}$ together with a unipotent conjugacy class in the centralizer $Z_{\GL{n}{k}}(x^{ss})$. \\
Let $p(T) = \prod_i f_i(T)^{n_i}$ be the characteristic polynomial of $x^{ss}$ and denote $d_i = \deg f_i$ the degrees of the irreducible factors $f_i$, then we have an isomorphism $Z_{\GL{n}{k}}(x^{ss}) \cong \prod_i \GL{n_i}{k_{d_i}}$. In general, by the theory of Jordan blocks a unipotent conjugacy class in $\GL{m}{k}$ is uniquely determined by a partition of $m$, so that a unipotent conjugacy class in $Z_{\GL{n}{k}}(x^{ss})$ is determined by a collection of partitions: for each polynomial $f_i$ one gets a partition of $n_i$. \\
Finally, $\nu_x(f) = 0$ is the empty partition if $f$ is not an irreducible factor of $p(T)$, while if $f = f_i$, $\nu_x(f_i)$ is the partition of $n_i$ associated to $f_i$. \\
Recall now the definition of $\rho ( \lambda, t ; f)$ (see also section 2.5 of \cite{S2}) for a general partition $\lambda = (n_1, \ldots, n_r)$: here $t \in T_{\lambda} \cong \prod_i k_{n_i}^*$, so one has that $t$ corresponds to $(x_1, \ldots, x_r) \in \prod_i k_{n_i}^*$, and defines \[ r_j(t,f) = \left| \left\{ 1 \le i \le r \textnormal{ such that } n_i = j \cdot \deg f \textnormal{ and } f(x_i) = 0 \right\} \right|. \]
Then $\rho(\lambda, t ; f)$ is the partition with $r_j(t,f)$ parts equal to $j$. \\
Let's specialize the above to the case $\lambda = \{ n \}$. One obtains that \[ r_j(t,f) = \left| \{ i \le 1 \, | \, n=n_1 = j \deg f \textnormal{ and } f(t)=0 \} \right|, \] so if $f=m$ is the minimal polynomial of $t$ we have \[ r_j(t,m) = \left\{ \begin{array}{cc} 1 & \textnormal{ if } j = \frac{n}{\deg m} \\ 0 & \textnormal{ otherwise} \end{array} \right. \]
and we get that $\rho( \{ n \}, t ; m) = \left\{ \frac{n}{\deg m} \right\}$ is a partition of $\frac{n}{\deg m}$. For every other $f \neq m$, the partition $\rho(\{n\}, t ; f) = 0$ is the empty partition, and then the Green polynomial yields $Q( \nu_x(f), 0) = 1$, as in proposition \ref{greenpolys}. \\
We obtain thus \[ H_{k_l} \left( x; \{ n \} , t \right) = Q \left( \nu_x(m), 
\left\{ \frac{n}{\deg m} \right\} \right) \left( {|k_l|}^{\deg m} \right), \] and similarly, \[ H_k \left( x; \{ n \} , t \right) = Q \left( \nu_x(m), \left\{ \frac{n}{\deg m} \right\} \right) \left( {|k|}^{\deg m} \right). \]
Since $Q \left( \nu_x(m), \left\{ \frac{n}{\deg m} \right\} \right) = Q(T) \in \Z[T]$ is a polynomial with integer coefficients, it turns out that after reduction mod $l$ we obtain \[ \overline{H_{k_l} \left( x; \{ n \} , t \right)} = \overline{Q \left( |k_l|^{\deg m} \right) }= \overline{ Q \left( |k|^{l \cdot \deg m} \right)  }= \overline{Q \left( |k|^{deg m} \right)}^l = \overline{ H_k \left( x; \{ n \} , t \right) }^l. \]
It remains to compare the sums over the Weyl-conjugates. Notice that since $t \in k_n^*$, $\Norm{k_{ln}}{k_n} (t) = t^l$ and the same is true for its Weyl-conjugates which are just powers of $t$. Hence \[ \sum_{w \in W_n} \phi(w.t) = \sum_{w \in W_n} \psi(w.t)^l = \sum_{i=1}^n \psi \left( t^{|k_l|^i}\right)^l = \sum_{i=1}^n \psi \left( t^{|k|^{il}}\right)^l. \] Since $(l,n)=1$ we in fact have that the two multisets $ \{ t^{|k|^{il}} \}_{i=1}^n $ and $\{ t^{|k|^j} \}_{j=1}^n $ are equal, hence \[ \sum_{i=1}^n \psi \left( t^{|k_l|^i}\right)^l = \sum_{j=1}^n \psi \left( t^{|k|^j}\right)^l = \sum_{w' \in W'_n} \psi \left( w'.t \right)^l \] which concludes the proof.
\end{proof}

To prove that Tate cohomology realizes the Shintani correspondence on the finite groups level, it remains to show that the reduction $\mod l$ of an irreducible $l$-adic representation $\rho$ of $\GL{n}{k_l}$, considered as a modular $\GL{n}{k}$-representation, is exactly the Tate cohomology $T^0( \rho)$. This follows from the next two results:
\begin{thm}\label{Tatechar} Let $ \rho $ be an irreducible $l$-adic representation of $\GL{n}{k_l}$ which is $\Frob_k$-fixed and consider its Tate cohomology with respect to that action. The Tate cohomology of $\rho$ is naturally a $l$-modular representation of $\GL{n}{k}$ (as discussed in the introduction), and for the reduction $\mod l$ we have the following identity of traces: \[ \overline{\chi_{\rho}(g)} = \Tr(g|_{T^0}) - \Tr(g|_{T^1}) \qquad \forall g \in \GL{n}{k} \] where $\chi_{\rho}$ denotes the character of the representation $\rho$, and the bar denotes, as usual, reduction $\mod l$. For the finite order operator $g$ on the finite-dimensional $\overline \Fl$-vector space $T^i(\rho)$, we denote by $\Tr(g|_{T^i})$ the sum of the eigenvalues of $g$.
\end{thm}
\begin{proof} We reduce the question to a pure module-theoretic computation. Fix $g \in \GL{n}{k}$ for which we will prove the claim. 

Notice that every $l$-adic irreducible representation of $\GL{n}{k_l}$ is $l$-integral, because it is finite-dimensional (since $\GL{n}{k_l}$ is a finite group) and in particular an invariant lattice is obtained by averaging over $\GL{n}{k_l}$ any full-dimensional lattice. It is clear that $l$-integrality is in fact obtained over a finite extension of $\Ql$ (see, for example, \cite{V} Chapter II, section 4.12), and hence $\rho$ can be realized over a free $\Zl$-module $V$ of finite rank. \\
Setting now $G = \langle \Frob_k \rangle$ (a cyclic group of prime order $l$) gives $V$ the structure of a $\Zl[G]$ module, since $V$ is $\Frob_k$-stable. As the action of $g$ commutes with $\Frob_k$, $g$ is a $\Zl[G]$-module automorphism of finite order. \\
By finiteness of the rank and the fact that trace is additive in short exact sequences, it's enough to prove the claim for each indecomposable $\Zl[G]$-module, and following Curtis and Reiner in \cite{CR}, page 690, there are exactly 3 of them. Borchards in \cite{Borchards1}, section 2, shows that these three are $\Zl$, $\Zl[G]$, and the augmentation ideal $I \subset \Zl[G]$ which can be seen either as the kernel of the augmentation map $\Zl[G] \lra \Zl$ or as the quotient $\Zl[G] / N \Zl[G]$ with $N = 1 + \Frob_k + \ldots \Frob_k^{l-1}$ being the norm element. \\
In \cite{Borchards1} we are also given the Tate cohomology for these three indecomposable pieces: \[ \begin{array}{lll}  & T^0 \left( \Zl \right) = \Z / l \Z \qquad & T^1 \left( \Zl \right) = 0, \\ & T^0 \left( \Zl[G] \right) = 0 \qquad & T^1 \left( \Zl[G] \right) = 0, \\ & T^0 \left( I \right) = 0 \qquad & T^1 \left( I \right) = \Z / l \Z. \end{array} \]
Consider first $V= \Zl$, then obviously $g$ acts as a unit $\xi \in \Zl^*$ and its reduction $ \overline \xi$ coincides with the action of $\xi$ on $T^0(\Zl) = \Z / l \Z$. Since $T^1(\Zl)=0$, the claim of the theorem holds in this case. \\
Consider now $V = \Zl[G]$, then $g \in \GL{n}{k}$ is a finite order $\Zl[G]$-automorphism of $V$. Since $V$ is indecomposable, $g$ admits a unique eigenvalue $\lambda \in \overline {\Zl}^*$, hence its trace on $V$ is $\rk_{\Zl} (V) \cdot \lambda = l \cdot \lambda$ whose reduction $\mod l$ is zero. \\
Finally, consider $V = I = \Zl[G] / N \Zl[G]$. Again $g$ acts on $V$ as a $\Zl[G]$-automorphism of finite order, hence since $V$ is indecomposable $g$ admits a unique eigenvalue $\lambda$. Then its trace on $V$ is $\rk_{\Zl} (V) \cdot \lambda = (l-1) \lambda$ whose reduction is $ - \overline{\lambda}$. But since $T^1(V) = \Z / l \Z$ is a subquotient of $V$ it is clear that $g$ will act as multiplication by $\overline \lambda$ on it, which proves the claim.
\end{proof}

We record here for sake of clarity the result obtained on the finite group level, which follows immediately from theorems \ref{charscorresp}, \ref{Tatechar} and \ref{T1dies}.
\begin{thm} Let $l, p$ be different primes, $n$ be a positive integer coprime to both and $k$ be a finite field of characteristic $p$. Let $\pi$ be a cuspidal $l$-adic representation of $\GL{n}{k}$ and $\rho = \mathbf{bc} ( \pi )$ its Shintani base change to $\GL{n}{k_l}$. Then \[ r_l(\pi)^{(l)} \cong T^0 (\rho) \quad \textnormal{ as $l$-modular } \GL{n}{k} \textnormal{-representations} \]
\end{thm}

\section{Lifting to $p$-adic groups}
In this section we lift the cuspidal representations of $\GL{n}{k}$ constructed so far to level zero cuspidal representations of $\GL{n}{L}$ where $L$ is a finite extension of $\Qp$ with residue field $k$. There is a standard construction to do so, and we recall it here.

Denote then $k = \F{q}$, a finite field of characteristic $p$, and let $L/ \Qp$ be a finite extension whose residue field is $k$ (we make no assumption on the absolute ramification degree).
\begin{thm} \label{cusplift} Let $\rho$ be a cuspidal representation of $\GL{n}{k}$ over either $R = \overline{\Ql}$ or $R = \FFl$, with central character $\omega_{\rho}: k^* \lra R^*$. Using the canonical quotient map $\GL{n}{\OO_L} \lra \GL{n}{k}$, consider $\rho$ as a representation of $\GL{n}{\OO_L}$, and similarly lift $\omega_{\rho}$ to $\OO_L^*$ via the quotient map $\OO_L^* \lra k^*$. Using the isomorphism given by valuation $L^* \cong \OO_L^* \times \Z$, extend $\omega_{\rho}$ trivially on $\Z$ so that we have $\omega: L^* \lra R^*$. Then
\begin{enumerate}
\item the induced representation \[ \pmb \rho = \Ind^{\GL{n}{L}}_{L^* \GL{n}{\OO_L}} \omega \rho \] is irreducible and cuspidal;
\item moreover, $\pmb \rho$ coincides with the compactly-induced representation $\mathrm{c-ind}^{\GL{n}{L}}_{L^* \GL{n}{\OO_L}} \omega \rho$;
\item $\pmb \rho|_{\GL{n}{\OO_L}}$ contains $\rho$ with multiplicity one;
\item all cuspidal representations of level zero and minimal-maximal type of $\GL{n}{L}$ are obtained in this way.
\end{enumerate}
\end{thm}
\begin{proof} This is theorem 3.3 in chapter III of \cite{V}.
\end{proof}
From now on, we will call the representation $\pmb \rho$ (obtained as in the theorem) the \emph{lift} of $\rho$.
\begin{rem} The extension of $\omega_{\rho}$ over $\Z$ to a character $\omega$ of $L^*$ is arbitrary, that is, any choice of an isomorphism $L^* \cong \OO_L^* \times \Z$ - which corresponds to a choice of a uniformizer in $L^*$ that splits the obvious short exact sequence -  is allowed and give another irreducible cuspidal representation of $\GL{n}{L}$. It is possible to describe when two different choices give rise to isomorphic representations, but it is not of interest now (see again theorem 3.3 in chapter III of \cite{V}).
\end{rem}

We can then apply theorem \ref{cusplift} to our current situation. Let $F \supset E$ be a degree $l$ unramified extension of local fields of characteristic 0 and residue characteristic $p$. Fix a common uniformizer $\varpi$ which pins down isomorphisms $F^* \cong \OO_F^* \times \Z$ and $E^* \cong \OO_E^* \times \Z$. \\
Let $\pi$ be an integral $l$-adic cuspidal representation of $\GL{n}{k}$ and $\rho$ the correspondent Frobenius-fixed cuspidal representation of $\GL{n}{k_l}$. We know that the Frobenius twist of the reduction mod $l$ $r_l (\pi)$ is isomorphic to $ \tau = T^0(\rho)$, as modular $\GL{n}{k}$-representations. \\
We lift all four representations through the process described in theorem \ref{cusplift}, using the isomorphisms provided by the fixed choice of $\varpi$, obtaining representation of $p$-adic reductive groups. We denote the new representations by using the same letters in a bold font: respectively $\pmb \pi$ being an integral $l$-adic representation of $\GL{n}{E}$ lifting $\pi$ and $ \pmb{ r_l (\pi) }$ being the lift of its reduction; $\pmb \rho$ is the corresponding integral $l$-adic representation of $\GL{n}{F}$ and the lift of the modular representation realized via Tate cohomology is denoted by $\pmb \tau$. Thanks to the common choice of uniformizers, $\pmb \pi$ and $\pmb \rho$ have the same central character restricted to $E^*$, and the same holds for the corresponding modular representations.

First of all, we want to check that the lifting procedures commute with the correspondences that we had on the level of finite groups, that is
\begin{prop}\label{compat} Assume the notation above.
\begin{enumerate}
\item The lift of the reduction $\pmb{r_l (\pi)}$ is isomorphic to the reduction of the lift $r_l ( {\pmb \pi})$.
\item The lift of a Frobenius twist is the Frobenius twist of the lift.
\item For the Frobenius-fixed $l$-adic representation $\rho$, lifting commutes with taking the Tate cohomology $T^0$, that is $\pmb \tau \cong T^0(\pmb \rho)$, where the Tate cohomology $T^0$ of a cuspidal representation of $p$-adic reductive groups is appropriately defined as in \cite{TV}, section 3.
\end{enumerate}
\end{prop}
\begin{proof}
\begin{enumerate}
\item 
Denote $M$ the underlying space of $\pi$, and fix a $\GL{n}{k}$-stable lattice $L \subset M$. Then $L$ is $\GL{n}{\OO(E)}$-stable under the canonical surjection $\GL{n}{\OO(E)} \twoheadrightarrow \GL{n}{k}$, and also $E^*$-invariant, where $E^* \cong \OO(E)^* \times \varpi^{\Z}$ has the second factor acting trivially as explained above.

We obtain then that  \[ \mathcal L = \Ind^{\GL{n}{E}}_{E^* \GL{n}{\OO(E)}} L = \left\{ \GL{n}{E} \stackrel{f}{\lra} L \textnormal{ compactly supported and } f(zkg) = \omega_{\pi}(z) \pi(k). f(g) \right\} \] is a $\GL{n}{E}$-stable lattice inside \[ \Ind^{\GL{n}{E}}_{E^* \GL{n}{\OO(E)}} M = \left\{ \GL{n}{E} \stackrel{f}{\lra} M \textnormal{ compactly supported and } f(zkg) = \omega_{\pi}(z) \pi(k). f(g) \right\}. \]
Since we have a stable lattice for the lift $\pmb \pi$, which is irreducible and cuspidal by theorem \ref{cusplift}, theorem 1.1d in chapter 3 of \cite{V} guarantees that the reduction $\mod l$ $r_l(\pmb \pi)$ is irreducible. In particular its underlying space is precisely $\mathcal L / \mathfrak m \mathcal L$ where we denote by $\mathfrak m$ the maximal ideal of $\overline \Zl$.

On the other hand, again by theorem 1.1d in chapter 3 of \cite{V}, the representation $\pi$ has irreducible reduction $\mod l$, so that applying theorem \ref{cusplift} to the $l$-modular representation $r_l(\pi)$ yields the irreducible, cuspidal $\GL{n}{E}$-representation $\pmb{r_l(\pi)}$. Notice that by irreducibility the underlying space of $r_l(\pi)$ is $L / \mathfrak m L$, and hence the underlying space of the lift $\pmb{r_l(\pi)}$ is \[ \Ind^{\GL{n}{E}}_{E^* \GL{n}{\OO(E)}} r_l(\pi) = \left\{ \GL{n}{E} \stackrel{f}{\lra} L / \mathfrak m L \textnormal{ compactly supp. and } f(zkg) = \omega_{r_l(\pi)}(z) r_l(\pi)(k). f(g) \right\}. \]

Since both $r_l(\pmb \pi)$ and $\pmb{r_l(\pi)}$ are irreducible, to prove that they are isomorphic it suffices to give a $\GL{n}{E}$-equivariant nonzero map \[ \Delta: r_l(\pmb \pi) \lra \pmb{r_l(\pi)}. \]
Take $f \in \mathcal L = \left\{ \GL{n}{E} \stackrel{f}{\lra} L \textnormal{ compactly supported and } f(zkg) = \omega_{\pi}(z) \pi(k). f(g) \right\}$, then we define $\Delta(f)$ to be the composition \[ \GL{n}{E} \stackrel{f}{ \lra} L \twoheadrightarrow L / \mathfrak m L. \]
Indeed, it is immediate that this is well-defined on $\mathcal L / \mathfrak m \mathcal L$, since \[ \mathfrak m \mathcal L = \left\{ \GL{n}{E} \stackrel{f}{\lra} \mathfrak m L \textnormal{ compactly supported and } f(zkg) = \omega_{\pi}(z) \pi(k). f(g) \right\} , \] and notice that under the quotient $L \twoheadrightarrow L / \mathfrak m L$, each $\omega_{\pi}(z) \pi(k) \in \mathrm{GL}(L) \twoheadrightarrow \mathrm {GL} \left( L / \mathfrak m L \right) $ maps to $\omega_{r_l(\pi)}(z) r_l(\pi)(k)$. Finally, the fact that the $\GL{n}{E}$-actions are compatible is clear by staring at the explicit function models.

\item Recall that the Frobenius twist of any modular $l$-representation $(\phi, V)$ can be defined as \[ \left( \phi, V \otimes_{ \left( \FFl , \Frob_l \right) } \FFl  \right) \textnormal{ with the action } \phi(g). (v \otimes x) = \phi(g).v \otimes x \] where the $\FFl$-action on the copy of $\FFl$ is via Frobenius, that is \[ v \otimes x^ly = xv \otimes y \quad \forall v \in V, \, \forall x,y \in \FFl. \] Thus, to prove the isomorphism stated, it suffices to show that \[ \pmb \alpha \otimes_{\left( \FFl , \Frob_l \right)} \FFl \cong \pmb{ \alpha \otimes_{\left( \FFl , \Frob_l \right)}  \FFl} \] for any $l$-modular representation $( \alpha, V)$ of $\GL{n}{k}$. \\
The underlying space of the left hand side is \[ \pmb V \otimes_{\left( \FFl , \Frob_l \right)} \FFl = \{ f: \GL{n}{E} \lra V \textnormal{ s. t. } f(hg) = \alpha(h).f(g) \, \forall h, \, \forall g \in \GL{n}{E} \} \otimes_{\left( \FFl , \Frob_l \right)} \FFl \] while the underlying space of the right hand side is \[ \pmb { V \otimes_{\left( \FFl , \Frob_l \right)} \FFl } = \{ f: \GL{n}{E} \lra V \otimes_{\left( \FFl , \Frob_l \right)} \FFl \textnormal{ s. t. } f(hg) = \alpha(h).f(g) \, \forall h, \, \forall g \in \GL{n}{E} \}. \] The map \[ \Delta: \pmb V \otimes_{\left( \FFl , \Frob_l \right)} \FFl \lra \pmb { V \otimes_{\left( \FFl , \Frob_l \right)} \FFl } \textnormal{ where } \Delta \left( f \otimes \lambda \right) (g) = f(g) \otimes \lambda \] is well-defined on the tensor product, and clearly respects the $\GL{n}{E}$-action, since this action is trivial on the copy of $\FFl$. $\Delta$ is obviously an isomorphism, hence the proof is complete.

\item We will use Bernstein and Zelevinsky' characterization of induced representations as $l$-sheaves (see \cite{BZ2}), together with the proposition in section 3.3 of \cite{TV} which we now recall without proof.
\begin{prop}\label{TVprop} Let $X$ be an $l$-space and $\mathcal F$ be an $l$-sheaf on $X$. Suppose $\sigma$ is a finite-order automorphism of $X$ and that $\mathcal F$ is $\sigma$-equivariant, so that one can define Tate cohomology as \[ T^0(\mathcal F|_{X^{\sigma}}) : = \ker(1- \sigma) / \im N \] which is an $l$-sheaf on $X^{\sigma}$.
We then have a canonical isomorphism for compactly supported sections: \[ T^0 \left( \Gamma_c(X, \mathcal F) \right) \stackrel{\cong}{\lra} \Gamma_c \left( X^{\sigma} , T^0(\mathcal F) \right) \] where the map is given on a compactly supported section of $\mathcal F$ on $X$ by restricting it to $X^{\sigma}$.
\end{prop}
\begin{lem} Denote now $G= {\GL{n}{F}}$ and $H = F^* \GL{n}{\OO_F}$, a closed subgroup of $G$. Let $\sigma = \Frob_E$ act on $G$, clearly with order $l$. Then \[ \left( H \backslash G \right)^{\sigma} \cong H^{\sigma} \backslash G^{\sigma} \] canonically.
\end{lem}
\begin{proof} This is a cohomology computation,: we need to show that \[ \ker \left( H^1( \sigma, F^* \GL{n}{\OO_F}) \lra H^1( \sigma, \GL{n}{F} \right) = 1. \] In fact, we will show that the cohomology group on the left is trivial. Since $F / E$ is unramified, we can pick a common uniformizer $\omega$, and then $F^* \cong \omega^{\Z} \times \OO_F^* $ so that $F^* \GL{n}{\OO_F} = \omega^{\Z} \GL{n}{\OO_F}$.
The long exact sequence in group cohomology gives \[ \ldots \lra H^1(\sigma, \omega^{\Z}) \lra H^1(\sigma, F^* \GL{n}{\OO_F}) \lra H^1 (\sigma, \GL{n}{\OO_F}) \lra \ldots \] but the Frobenius action on $\omega^{\Z}$ is trivial, hence $H^1(\sigma, \omega^{\Z} ) \cong \Hom \left( \Z / l \Z , \Z \right) =0 $. Thus, it suffices to show that $H^1 (\sigma, \GL{n}{\OO_F})=0$. \\
Consider now the first congruence subgroup $\Gamma = \id_n + \omega \mathrm{Mat}_n(\OO_F) \subset \GL{n}{\OO_F}$. This is a Frobenius-invariant, normal pro-$p$ group and the quotient is $\GL{n}{\OO_F} / \Gamma \cong \GL{n}{k_l}$.
The long exact sequence in cohomology thus gives \[ \ldots \lra H^1(\sigma, \Gamma) \lra H^1(\sigma, \GL{n}{\OO_F}) \lra H^1 (\sigma, \GL{n}{k_l}) \lra \ldots \] and since $(p,l)=1$, the $\langle \sigma \rangle$-cohomology of the pro-$p$ group $\Gamma$ is trivial. It remains to show that $H^1(\sigma, \GL{n}{k_l}) = 0 $, but this is a well-known generalization of Hilbert 90.
\end{proof}
Consider now the $l$-adic representation $\pmb \rho = \cInd^G_H \omega_{\rho} \rho$: by proposition 2.23 in \cite{BZ2} there exists a unique (up to isomorphism) $l$-sheaf $\mathcal F$ on $X = H \backslash G$ such that $\pmb \rho$ is isomorphic to the space of compactly supported sections $\Gamma_c(X, \mathcal F)$ as modules over the algebra of $C_c^{\infty}(X)$ of compactly supported, locally constant functions on the $l$-space $X$ (and by equivalence of categories, they are isomorphic as admissible $G$-representations).
\begin{lem} The compactly supported sections of $\mathcal F$ on an open set $U \subset X$ are given by \[ \Gamma_c(U, \mathcal F) = \{ f \in \cInd^G_H W \textnormal{ such that } \supp f \subset HU \}.  \]
\end{lem}
\begin{proof} We follow the proof of proposition 1.14 in \cite{BZ2}. Notice that since $H \subset G$ is open, $X = H \backslash G$ has the discrete topology and hence functions in $ C_c^{\infty}(X)$ and sections in $ \Gamma(X, \mathcal F)$ are compactly supported if and only if are finitely supported modulo $H$. \\
In particular, a compactly supported section of $\mathcal F$ on an open set $U \subset X$ is a section with finite support contained in $U$; therefore the claim of the lemma follows from the special case $U = Hg = x \in X$, i.e. we want to show that the stalk $\mathcal F_x$ is given by elements $f \in \cInd^G_H W$ supported at $Hg$. \\
The explicit construction of proposition 1.14 in \cite{BZ2} gives that for the $C_c^{\infty}(X)$-module $M = \cInd^G_H W$, the associated sheaf $\mathcal F$ has stalk \[ \mathcal F_x = M / M(x) \textnormal{ where } M(x) = \left\{ f \in \cInd^G_H W \, | \, f(g) =0 \right\} \]
Thanks to the topology being discrete, it is immediate that $M \twoheadrightarrow M /M(x)$ has a canonical section whose image consists of the elements $f \in \cInd^G_H W$ supported at $Hg$, and this concludes the proof.
\end{proof}
Since we assume that $\pmb \rho$ is Frobenius-fixed, we get an induced action by $\sigma = \Frob$ on the sheaf $\mathcal F$, and $\mathcal F$ is $\sigma$-equivariant. Explicitly, for an open set $U \subset H \backslash G$, an element $f \in \mathcal F(U)$ and any $g \in G$ the action is \[ \sigma(U): \mathcal F(U) \lra \mathcal F(\sigma^{-1}(U)), \quad \left( \sigma f \right) (g) = \sigma^{-1} \left( f(\Frob(g)) \right) \] where the $\sigma$-action on the right-hand side consists of the $\sigma$-action on the underlying space $W$ of $\omega_{\rho} \rho$. \\
We are then able to explicitly determine $\mathcal F^\sigma$: letting $U \subset \left( H \backslash G \right)^{\sigma}$ open, we have \[ \ker(1 - \sigma)(U) = \{ f \in \mathcal F(U) \textnormal{ such that } f(\Frob(g)) = \sigma(f(g)) \, \forall g \in U \}; \] but since $U$ is $\sigma$-fixed, the latter condition corresponds to $f(g) \in W^{\sigma}$ for all $g$, which means $\im f \subset W^{\sigma}$. Similarly, $N = \sum_{i=1}^l \sigma$ acting on $\mathcal F(U)$ gives \[ \left( N.f \right)(g) = N (f(g)), \] hence \[ \left( \im N \right)(U) = \{ (N.f) \textnormal{ such that } f \in \mathcal F(U) \} \] where now $(N.f)(g) = N(f(g)) \in \im N \subset W^{\sigma}$. \\
Notice that for any fixed open set $U \subset H^{\sigma} \backslash G^{\sigma}$, we can find a function $f \in \ker \left( 1 - \sigma \right) (U)$ which hits any specific element of $W^{\sigma}$ since those functions are locally constant and compactly supported. Hence by the lemma \[ \Gamma_c \left( X^{\sigma} , T^0(\mathcal F) \right) = \{ f \in \Ind^{G^{\sigma}}_{H^{\sigma}} T^0(W) \textnormal{ compactly supported} \} = \cInd^{\GL{n}{E}}_{E^* \GL{n}{\OO_E}} T^0(W) \] which concludes the proof of proposition \ref{compat}, (3) by applying the isomorphism of proposition \ref{TVprop}.
\end{enumerate}
\end{proof}

The remaining question is, what does the Shintani correspondence on the finite groups level become, when lifted through theorem \ref{cusplift}?
\begin{thm} \label{Shintaniunram} The lifting of the Shintani correspondence to the $p$-adic groups level realizes base change (defined as in \cite{AC}, chapter 1, section 6).
\end{thm}
\begin{proof}[Proof of theorem \ref{Shintaniunram}]
We will use an explicit description of the local Langlands correspondence to prove this statement. \\
One of the features that characterize uniqueness of the local Langlands correspondence is that base change on the automorphic side corresponds to restriction on the Galois side, that is: if $ \phi$ is an irreducible admissibile representation of $\GL{n}{E}$ corresponding to the $n$-dimensional representation $\alpha$ of the Weil group $W(E)$, and $F/E$ is a cyclic extension of local fields, then the restriction $\res^{W(E)}_{W(F)} \alpha$ corresponds via Local Langlands to the base change $\mathbf{bc} (\phi)$, an irreducible admissibile representation of $\GL{n}{F}$. We remark again that this description, which is classical for representation over a characteristic zero coefficient field, still holds true for modular representations (at least in characteristic different from the residue characteristic of $E$) thanks to Vigneras (\cite{Vig}).

Therefore, it is enough to find the Galois representations corresponding to $\pmb \pi$ and $\pmb \rho$ and check that they are the restriction of one another. We follow here the explicit description by Bushnell and Henniart (see \cite{BH1} and \cite{BH2}): the information that we need is encoded in the following lemma. We keep the notation introduced at the start of this section: $k = \F{q}$ is a finite field of characteristic $p$ and $L / \Qp$ is a finite extension having residue field $k$.
\begin{adjustwidth}{1cm}{0pt}
\begin{lem}\label{LLC} Suppose $(n,|k|)=1$. Let $\alpha$ be an $l$-adic cuspidal representation of $\GL{n}{k}$ which corresponds via Green correspondence to a regular character $\chi: k_n^* \lra \overline \Ql^*$. Inflate $\chi$ to $\OO_{L_n}^*$ and by extending it trivially on $\Z$ get a character (still denoted by $\chi$) of $L_n^*$. Consider the unramified character $\mu: \L_n^* \lra \overline \Ql^*$ defined as $\mu(\omega) = (-1)^{n-1}$ on a (equivalently, any) uniformizer $\omega$ of the unramified extension $L_n$ of degree $n$ of $L$. Then the ($l$-adic) local Langlands correspondence gives, for the level zero cuspidal $l$-adic representation obtained as in theorem \ref{cusplift}, \[ \pmb \pi = \Ind^{\GL{n}{L}}_{L^* \GL{n}{\OO_L}} \omega_{\alpha} \alpha \longleftrightarrow \Ind^{W(L)}_{W(L_n)} \mu \chi \] where on the Galois side, $\mu \chi$ becomes a character of $W(L_n)$ via the quotient map to the abelianization $W(L_n)^{ab} \cong L_n^*$ using the canonical isomorphism of class field theory.
\end{lem}

\begin{proof} Recall that we fixed an isomorphism $\overline \Q_l \cong \C$.
Bushnell and Henniart (\cite{BH1}) parametrize both sets of the local Langlands correspondence ($n$-dimensional Weil representations of $W(L)$ and irreducible, admissibile representations of $\GL{n}{L}$ for $L$ a non-archimedean local field of characteristic $0$) by the set of \emph{admissible pairs}, that is pairs $(L'/L, \xi)$ where $L'/L$ is tamely ramified and $\xi: L'^* \lra \overline \Q_l^*$ is a (continuous) character such that for every intermediate extension $L' \supset K \supset L $
\begin{enumerate}
\item if $\xi$ factors through $\Norm {L'} {K}$, then $L'=K$; 
\item if $\xi|_{ \left( 1 + \mathfrak p_{L'} \right)}$ factors through $\Norm{L'} {K}$, then $L' \supset K$ is unramified.
\end{enumerate}
Suppose $L'/L$ is unramified. Then the correspondence between admissible pairs and Weil representations simply associates to $(L'/L, \xi)$ the $[L':L]$-dimensional Weil representation $\Ind^{W(L)}_{W(L')} \xi$ where $\xi$ is extended to be a character of $W(L')$ via the quotient map to its abelianization $W(L')^{ab}$ and the isomorphism of class field theory. \\
Suppose now $\pmb \pi$ is a level zero cuspidal $l$-adic representation of $\GL{n}{L}$, with notation as in the statement of the lemma. Then proposition 2.2 in \cite{BH1} says that associating to $\pmb \pi$ the admissible pair $(L_n / L , \chi)$ is part of a more general bijection between admissible pairs and irreducible admissibile $\GL{n}{L}$-representations: the composition of the two bijections just described gives rise to the \emph{naive correspondence}. We denote this bijection as \[ (L'/L, \xi) \longleftrightarrow _L\pmb \pi_{\xi}. \]
The modification needed to turn the composition of the two bijections described above into the local Langlands correspondence consists in using a \emph{rectifier}, that is a tamely ramified character $\mu: L'^* \lra \overline \Q_l^*$ depending on the admissible pair $(L' / L, \xi)$, such that the new bijection \[ (L'/L, \xi) \longleftrightarrow _L\pmb \pi_{\mu \xi} \] between admissible pairs and $\GL{n}{L}$-representations yields the local Langlands correspondence upon composition: \[ \Ind^{W(L)}_{W(L')} \xi \longleftrightarrow _L\pmb \pi_{\mu \xi}, \] as in theorem A (or corollary 3.3) of \cite{BH1}.

Therefore, it remains to figure out what the rectifier looks like in case of a level zero cuspidal representation $\pmb \pi$: in this situation, admissibility of the associated pair $(L'/L, \xi)$ is equivalent to $L'/L$ being unramified, $\xi |_{1 + \mathfrak p_{L'}} \equiv 1$ being trivial, and $\xi$ being $\Gal(L'/L)$-regular, as explained in section 2.2 of \cite{BH1}. Thus we set $L' = L_n$ from now on. \\
By theorems C and D in \cite{BH2}, the rectifier $\mu$ is then the product of the rectifier for the pair $(L_n/L_n, \xi)$ and the $u$-rectifier for $(L_n/L, \xi)$. The former is just the trivial character, since the Langlands correspondence must match with class field theory in the $n=1$ case. Theorem 2 of \cite{BH2} and uniqueness of the $u$-rectifier as a tamely ramified character (theorem C of the same paper) together imply that the unramified character of $L_n^*$ which sends a uniformizer $\omega \mapsto (-1)^{n-1}$ is exactly the $u$-rectifier we are looking for. The claim of lemma \ref{LLC} follows. 
\end{proof}
\end{adjustwidth}

We continue then the proof of theorem \ref{Shintaniunram} by applying the lemma to our situation: $\pmb \pi$ is a representation of $\GL{n}{E}$ lifted as in theorem \ref{cusplift} from a cuspidal representation of $\GL{n}{k}$ corresponding to the character \[ \psi: k_n^* \lra \overline \Ql^*, \] while $\pmb \rho$ is a representation of $\GL{n}{F}$ lifted from the cuspidal representation of $\GL{n}{k_l}$ corresponding to the character \[ \phi = \psi \circ \Norm{k_{ln}}{k_n} : k_{ln}^* \lra \overline \Ql^*. \]
Then by the lemma, the local Langlands correspondence gives \[ \pmb \pi \longleftrightarrow \Ind^{W(E)}_{W(E_n)} \mu \psi \] and \[ \pmb \rho \longleftrightarrow \Ind^{W(F)}_{W(F_n)} \mu \phi \] where $\mu$ is the unramified character of $F_{n}^*$ taking value $(-1)^{n-1}$ on any uniformizer. We hence need to check that \begin{equation} \label{llcweilside} \res^{W(E)}_{W(F)} \left( \Ind^{W(E)}_{W(E_n)} \mu \psi \right) \cong \Ind^{W(F)}_{W(F_n)} \mu \phi. \end{equation}
First of all notice that both $\mu \psi$ and $\mu \phi$ factors through a finite quotient (by construction), hence we can apply Mackey theory: a well known result (for example, see Serre - proposition 22 in chapter 7 of \cite{serre}) says that \[ \res^{W(E)}_{W(F)} \left( \Ind^{W(E)}_{W(E_n)} \mu \psi \right) \cong \bigoplus_{a \in W(F) \backslash W(E) /W(E_n) } \Ind^{W(F)}_{W(F) \cap W(E_n)^a} \left( \mu \psi \right)^a|_{W(F) \cap W(E_n)^a} \] where the $a$-superscript indicates conjugation. \\
Since $(l,n)=1$ and all extensions of local fields considered are unramified, we have \[ W(F) \backslash W(E) /W(E_n) \cong l \Z \backslash \Z / n \Z \] which clearly consists of a unique coset. Hence the right hand side of the Mackey isomorphism becomes \[ \Ind^{W(F)}_{W(F) \cap W(E_n)} \left( \mu \psi \right)|_{W(F) \cap W(E_n)} = \Ind^{W(F)}_{W(F_n)} \left( \mu \psi \right)|_{W(F_n)}. \] It remains to check that $\mu \psi |_{W(F_n)} \cong \mu \phi$. This follows by the commutativity of the following diagram in local class field theory:
\begin{displaymath}
\xymatrix{
        W(E_n) \ar[r]^{\textnormal{su}} & E_n^* \ar[r]^{\mu \psi} & \overline \Ql^*   \\
        W(F_n) \ar[r]^{\textnormal{su}} \ar@{^{(}->}[u] &  F_n^* \ar[u]_{\Norm{F_n}{E_n}} & }
\end{displaymath} where the horizontal surjective maps are compositions \[ W(L) \twoheadrightarrow W(L)^{ab} \stackrel{\cong}{\lra} L^* \] for each local field $L$.
Then we obtain $\mu \psi|_{{W(F_n)}} = \mu \psi \circ \Norm{F_n}{E_n}$; this obviously coincide with $\mu \phi$ on units of the ring of integers (by definition of $\psi$ and $\phi$ and since $\mu$ is unramified), while on a uniformizer $\omega \in E_n$ of $F_n$ we have \[ \mu \psi \circ \Norm{F_n}{E_n} (\omega) = \mu (\omega^l) = (-1)^{l(n-1)}. \] This should coincides with $\mu \phi(\omega) = (-1)^{n-1}$ and the only way that this fails to happen is when both $l$ and $n$ are even, which cannot be since we are assuming $(l,n)=1$. The proof is complete.
\end{proof}

\section{The ramified case}
In this section we investigate the ramified case. \\
Let $F \supset E$ be a tamely, totally ramified extension of local fields of prime degree $l$, where $E \supset \Qp$ is a finite extension. Denote by $k$ the common residue field. \\
Let $\pmb \pi$ be a level zero, minimal-maximal type (terminology as in \cite{V} , chapter III, section 3), cuspidal, $l$-adic representation of $\GL{n}{E}$ and $\pmb \rho = \pmb {\mathrm{bc}(\pi)}$ be its base change to $\GL{n}{F}$. Denoting by $\sigma$ a generator of $\Gal(F/E)$, the obvious action on $\GL{n}{F}$ turns $\pmb \rho$ into a $\sigma$-invariant representation - and thus we can take Tate cohomology with respect to this action. What can be said about the relation between the $l$-modular representations $T^0(\pmb \rho)$ and $r_l \left( \pmb \pi \right)$?

\begin{thm}\label{ramcase} With the notation as above, $T^0(\pmb {\mathrm{bc}(\pi)}) \cong r_l \left( \pmb \pi \right)^{(l)}$ as $l$-modular representations of $\GL{n}{E}$.
\end{thm}
Although the final result is the same as in the unramified case, the technique for this proof is radically different, since the two local fields have now the same residue field and we cannot hence hope to deduce the claim from the corresponding fact for finite groups (as in the unramified case).

By proposition 3.5, chapter II of \cite{FV} there exists a uniformizer $\varpi_E$ of $E$ such that some root of the Eisenstein equation $X^l - \varpi_E$ is a primitive element for the extension $F \supset E$. Fix one such root and denote it by $\varpi_F$: we have $F=E(\varpi_F)$ and $\OO_F = \OO_E (\varpi_F)$.

Denote $\pmb \pi = \Ind^{\GL{n}{E}}_{E^* \GL{n}{\OO_E}} \omega_{\alpha} \alpha$ where $\alpha$ is an $l$-adic representation of $\GL{n}{k_E}$ which corresponds (via the Green correspondence) to a regular character $\chi: (k_{E_n})^* \lra \overline \Zl^*$.
\begin{lem} For the $l$-adic base change $\pmb \rho$ of $\pmb \pi$, we have $\pmb \rho = \mathbf{bc} (\pmb \pi) = \Ind^{\GL{n}{F}}_{F^* \GL{n}{\OO_F}} \omega_{\alpha_l} \alpha_l$ where $\alpha_l$ is the representation of $\GL{n}{k_E}$ correpondent to the Green character $\chi^l: (k_{E_n})^* \lra \overline \Zl^*$. In particular, $\pmb \rho$ is also a cuspidal $l$-adic representation of level zero and minimal-maximal type.
\end{lem}
\begin{proof} This is a computation based on the explicit local Langlands by Bushnell and Henniart (\cite{BH1}, \cite{BH2}). Using lemma \ref{LLC} one gets the Weil representation corresponding to $\pmb \pi$: \[ \Ind^{W(E)}_{W(E_n)} \mu \chi \] where $\chi$ has been inflated to $\OO_{E_n}^*$ first, and then extended trivially to get a character of $E_n^*$, while $\mu : E_n^* \lra \overline \Zl^*$ is the unramified character taking value $(-1)^{n-1}$ on the uniformizer $\varpi_E$. Denote by $\gamma$ this Weil representation, and consider now $\res^{W(E)}_{W(F)} \gamma$. By Mackey theory, we get \[ \res^{W(E)}_{W(F)} \left( \Ind^{W(E)}_{W(E_n)} \mu \chi \right) \cong \bigoplus_{a \in W(F) \backslash W(E) /W(E_n) } \Ind^{W(F)}_{W(F) \cap W(E_n)^a} \left( \mu \chi \right)^a|_{W(F) \cap W(E_n)^a} \] where the $a$-superscript indicates conjugation. \\
Since $F$ and $E_n$ are linearly disjoint over $E$, we get that $W(F) W(E_n) = W(E)$ while $W(F) \cap W(E_n) = W(F E_n) = W(F_n)$ and hence the sum over the double cosets reduces to $\Ind^{W(F)}_{W(F_n)} \left( \mu \chi \right)|_{W(F_n)}$. The following commutative diagram from local class field theory \begin{displaymath}
\xymatrix{
        W(E_n) \ar[r]^{\textnormal{su}} & E_n^* \ar[r]^{\mu \chi} & \overline \Zl^*   \\
        W(F_n) \ar[r]^{\textnormal{su}} \ar@{^{(}->}[u] &  F_n^* \ar[u]_{\Norm{F_n}{E_n}} & }
\end{displaymath} tells us that \[ \left( \mu \chi \right)|_{W(F_n)} = \left( \mu \chi \right) \circ \Norm{F_n}{E_n}. \]
Now we want to show that the pair $ \left( F_n/F, ( \mu \chi) \circ \Norm{F_n}{E_n} \right)$ is admissible, in the terminology of \cite{BH1}. Denote $\xi = (\mu \chi) \circ \Norm{F_n}{E_n}$, then by Serre \cite{serre2}, chapter V, proposition 4 we have $\Norm{F_n}{E_n} \left( 1 + \varpi_F \OO_{F_n} \right) \subset 1 + \varpi_E \OO_{E_n}$ and since by construction $\mu \chi$ is trivial on the latter filtration subgroup, we get $\xi|_{ 1 + \varpi_F \OO_{F_n} } \equiv 1$. \\
As explained in section 2.2 of \cite{BH1}, admissibility of the pair $\left( F_n/F, \xi  \right)$ is then equivalent to $\xi$ being $\Gal (F_n / F)$-regular, that is, $\xi \neq \xi^{\sigma}$ for any nontrivial $\sigma \in \Gal(F_n / F)$. We show that this follows from the fact that $\chi$ is $\Gal(E_n / E)$-regular. Suppose $\xi$ is not regular, so that $\xi = \xi^{\sigma}$ for some nontrivial $\sigma \in \Gal(F_n / F)$. Therefore the two characters must coincide when restricted to $\OO_{F_n}^*$. \\
Notice that since $F$ and $E_n$ are linearly disjoint over $E$ and have compositum $F_n$, we have $\Gal(F_n / F) \times \Gal(F_n/E_n) \cong \Gal(F_n / E)$. In particular, $\sigma$ commutes with every element of $\Gal(F_n/E_n)$, so that \[ \sigma \left( \Norm{F_n}{E_n} (x) \right) = \Norm{F_n}{E_n} (\sigma(x)) \quad \forall x \in \OO_{F_n}^*. \]
By \cite{serre2}, chapter V, corollary 3, the image of $\Norm{F_n}{E_n} \left( \OO_{F_n}^* \right)$ contains $1 + \varpi_E \OO_{E_n}$ and by corollary 7 and the following remark we have that \[  k_n^*  \cong \OO_{E_n}^* / \left( 1 + \varpi_E \OO_{E_n} \right) \twoheadrightarrow \OO_{E_n}^* / \Norm{F_n}{E_n} \left( \OO_{F_n}^* \right) \cong \Z / l \Z \] is a prime order quotient of a cyclic group. \\
Since by definition $\mu$ is trivial on $\OO_{E_n}^*$, we obtain that $\chi = \chi^{\sigma}$ on an index $l$ subgroup of $k_n^*$. The character $\chi^\sigma \cdot \chi^{-1}$ is then a character of $\Z / l \Z$ of order dividing the order $r$ of $\sigma$, since clearly $\left( \chi^{\sigma} \right)^r = \chi^r$. But $r | n $, which is coprime to $l$, hence we have a character of $\Z / l \Z$ of order coprime to $l$. This order can only be 1, i.e. $\chi = \chi^{\sigma}$ on the entire $k_n^*$. But this contradicts the regularity of the character $\chi$, since $\sigma \in \Gal(F_n / F) \cong \Gal(k_n / k)$ is a non-trivial element.

Thus, the pair $\left( F_n / F, \left( \mu \chi \right) \circ \Norm{F_n}{E_n} \right)$ is admissible, and hence by proposition 2.2 in \cite{BH1}, the supercuspidal $l$-adic representation associated to the Galois representation $\Ind^{W(F)}_{W(F_n)} \xi$ is of level zero, minimal-maximal type, and ultimately depends on the Green character $\zeta:k_n^* \lra \overline \Zl^*$ whose inflation is $\xi$. To figure out what $\zeta$ is, we only need to compute $\xi$ on an element $x \in \OO_{F_n}^*$ whose reduction modulo $1 + \omega_F \OO_{F_n}$ generates $k_n^*$. But since $F_n / E_n$ is totally ramified, there exist one such element $x$ in $ \OO_{E_n}^*$: we get then \[ \xi(x) = \mu \chi \left( \Norm{F_n}{E_n} (x) \right) = \mu \chi (x^l) = \chi(x)^l, \] so that $\zeta = \chi^l$ and this ultimately shows the claim of the lemma.
\end{proof}

We can now compute Tate cohomology for the $l$-adic representation $\pmb \rho = \mathbf{bc} (\pmb \pi)$, we'll use proposition \ref{TVprop}. \\
In our situation, we have that $X = F^* \GL{n}{\OO_F} \backslash \GL{n}{F}$, $\sigma$ is a generator of $\Gal(F/E)$ and on an open set $U \subset F^* \GL{n}{\OO_F} \backslash \GL{n}{F}$ we have \[ \Gamma_c \left( U, \mathcal F \right) = \begin{array}{l} \{ f: \GL{n}{F} \lra  V  \textnormal{ supported on $F^* \GL{n}{\OO_F} U$ and such that} \\ f(zkg) = \omega_{\alpha_l}(z) \alpha_l(k).f(g) \quad \forall z \in F^*, \, k \in \GL{n}{\OO_F}, \, g \in \GL{n}{F} \}. \end{array} \] where we denote by $V$ the underlying space of $\alpha_l$. The map of sheaves realizing $\sigma$-equivariance of $\mathcal F$ is \[ \lambda: \mathcal F \lra \mathcal F \] defined on an open set $U \subset X$ as \[ \lambda(U): \mathcal F(U) \lra \mathcal F(\sigma^{-1}U) \quad \lambda(U).f (g) = f(\sigma g), \] because the identity map on $V$ is already realizing $\alpha_l \cong \alpha_l \circ \sigma$, due to the fact that $E \subset F$ is totally ramified.

Notice that $X = F^* \GL{n}{\OO_F} \backslash \GL{n}{F}$ is the set of vertices of the Bruhat-Tits building $\mathcal B \left( \GL{n}{F} \right) $ for $\GL{n}{F}$, which has a canonical structure of simplicial complex - and hence we can denote its set of vertices as $C^0 \left( \mathcal B \left( \GL{n}{F} \right)  \right)$. Each vertex corresponds to a homothety class of $\OO_F$-lattices in $F^n$, and the Bruhat-Tits building $\mathcal B \left( \GL{n}{E} \right) $ of $\GL{n}{E}$ is naturally embedded inside $\mathcal B \left( \GL{n}{F} \right) $ by considering homothety classes of $\OO_{F}$-lattices which have a basis defined over $\OO_E$. \\
The building admits a natural action of the Galois group, and it is a standard result (e.g. G. Prasad in \cite{Prasad}) that for tamely ramified extensions $F \supset E$ we have \[ \mathcal B \left( \GL{n}{F} \right)^{\Gal(F/E)}  = \mathcal B \left( \GL{n}{E} \right) \] where this is an equality of ``convex subsets" as explained in the source above. In particular, this means that the set of vertices $C^0 \left( \mathcal B \left( \GL{n}{F} \right)  \right)^{\Gal (F / E)}= X^{\sigma}$ is exactly the set of vertices of $\mathcal B \left( \GL{n}{F} \right)$ lying in the closure of a chamber of $\mathcal B \left( \GL{n}{E} \right)$ under the canonical embedding. \\
Consider now the canonical chamber for $\mathcal B \left( \GL{n}{E} \right)$: the one having vertices \[ \Lambda_1 = \left[ \OO_E \oplus \ldots \oplus \OO_E \right], \ldots, \Lambda_n = \left[ \mathcal \varpi_E \OO_E \oplus \ldots \oplus \varpi_E \OO_E \oplus \OO_E \right] \] and stabilized by the standard upper Iwahori. The vertices of $X^{\sigma}$ lying in the closure of this chamber are \[ \left\{ \textnormal{diag} \left( \varpi_F^{i_1}, \ldots, \varpi_F^{i_{n-1}}, 1 \right) . \Lambda_1 \right\} \textnormal{ for } l \ge i_1 \ge \ldots \ge i_{n-1} \ge 0. \] Since $\GL{n}{E}$ acts transitively on the chambers of $\mathcal B \left( \GL{n}{E} \right)$, we obtain that \[ X^{\sigma} = \bigcup_{l \ge i_1 \ge \ldots \ge i_{n-1} \ge 0} \GL{n}{E} \left\{ \textnormal{diag} \left( \varpi_F^{i_1}, \ldots, \varpi_F^{i_{n-1}}, 1 \right) . \Lambda_1 \right\} \] where the union is not disjoint.

\begin{prop} The sheaf $T^0( \mathcal F)$ is supported on $C^0 \left( \mathcal B \left( \GL{n}{E} \right)  \right)$, that is, supported on the vertex set of the building $\mathcal B \left( \GL{n}{E} \right) $ over the smaller field.
\end{prop}
The idea of the proof is the following: by topological considerations, one reduces to check that the sections of the sheaf $T^0( \mathcal F)$ over a vertex in the canonical chamber is zero unless this vertex is in $C^0 \left( \mathcal B \left( \GL{n}{E} \right) \right)$. For a vertex in the canonical chamber we explicitly compute $T^0$ which turns out to be sections of $\mathcal F$ taking values in the Tate cohomology of the underlying space $V$ with respect to a specific conjugation operator (depending on the vertex). A concrete description of $V$ given by the Kirillov model allows us to show that the Tate cohomology with respect to such an operator is trivial unless the conjugation itself is trivial, which corresponds to the vertex being in $C^0 \left( \mathcal B \left( \GL{n}{E} \right) \right)$.
\begin{proof}
First of all, since $F^* \GL{n}{\OO_F}$ is an open subgroup of $\GL{n}{F}$ the topology on $X$ is discrete, and hence it suffices to show that for any vertex $w \in C^0 \left( \mathcal B \left( \GL{n}{F} \right)  \right)^{\Gal(F/E)} - C^0 \left( \mathcal B \left( \GL{n}{E} \right) \right)$ we have $\Gamma(w, T^0 \left( \mathcal F \right) )=0$. Moreover, $\GL{n}{E}$ acts transitively on the chambers of its own building $\mathcal B \left( \GL{n}{E} \right)$ hence it suffices to show $\Gamma(w, T^0 \left( \mathcal F \right) )=0$ when $w$ is in the canonical chamber, because for any other $w'$ the representation induced by $T^0 \left( \mathcal F \right)$ at the vertex $w'$ will be $\GL{n}{E}$-conjugate (and hence $\GL{n}{E}$-isomorphic) to one given by a vertex in the canonical chamber.

Fix then a vertex \[ w = \diagon \left( \varpi_F^{i_1}, \ldots, \varpi_F^{i_{n-1}}, 1 \right) . \Lambda_1 \textnormal{ for some } l \ge i_1 \ge \ldots \ge i_{n-1} \ge 0 \] in the canonical chamber of $\mathcal B \left( \GL{n}{E} \right)$. We denote $w = \varpi_F^i. \Lambda_1$ and consider \[ \Gamma_c \left( \varpi_F^i. \Lambda_1 , \mathcal F \right) = \begin{array}{l} \{ f: \GL{n}{F} \lra  V  \textnormal{ supported on $F^* \GL{n}{\OO_F} \varpi_F^i$ and such that} \\ f(zkg) = \omega_{\alpha_l}(z) \alpha_l(k).f(g) \quad \forall z \in F^*, \, k \in \GL{n}{\OO_F}, \, g \in \GL{n}{F} \}. \end{array}. \] Any $f \in \Gamma_c \left( \varpi_F^i. \Lambda_1 , \mathcal F \right)$ is then obviously determined by its value on $\varpi_F^i$, and we denote $v = f(\varpi_F^i)$. Now $f \in \ker (1- \lambda)(w)$ if and only if $ \alpha_l \left( \xi_l^i \right) . v = v$, where we denote $\xi_l^i = \diagon \left( \xi_l^{i_1}, \ldots, \xi_l^{i_{n-1}}, 1 \right)$ for the $l$-th root of unity $\xi_l = \frac{\sigma(\varpi_F)}{\varpi_F}$. \\
Consider then \[ N(w) = \im \left( N_{\lambda} \right) (w) = \left\{ \sum_{m=0}^{l-1} \left( \lambda^m.f \right) \right\} \textnormal{ as $f$ varies in } \mathcal F(w). \] We have \[ \sum_{m=0}^{l-1} \left( \lambda^m.f \right) (\varpi_F^i) = \left( \sum_{m=0}^{l-1} \alpha_l^m(\xi_l^i) \right) . f(\varpi_F^i) \] and hence \[ N(w) = \left\{ f \in \mathcal F(w) \textnormal{ such that } v = f( \varpi_F^i) \in \im \left( \sum_{m=0}^{l-1} \alpha_l^m(\xi_l^i) \right) \right\}. \] We obtain \[ \Gamma_c \left( w, T^0(F) \right) = \begin{array}{l} \{ f: \GL{n}{F} \lra T^0(\alpha_l(\xi_l^i), V) \textnormal{ supported on } F^* \GL{n}{\OO_F}\varpi_F^i \\ \textnormal{ and such that } f(zkg) = \omega_{\alpha_l}(z) \alpha_l(k).f(g) \, \forall z \in F^*, \, k \in \GL{n}{\OO_F}, \, g \in \GL{n}{F} \}. \end{array} \]

It remains thus to compute $T^0(\alpha_l(\xi_l^i), V)$ and show that this is $0$ as soon as $w = \varpi_F^i \notin \mathcal C^0 \left( \mathcal B(\GL{n}{E}) \right)$. Notice that by character theory, \[ \alpha_l \cong \alpha_l \circ c_{\xi_l^i} \textnormal { as representations of } \GL{n}{k_E} \] where we denote by $c_g$ conjugation by $g \in \GL{n}{k_E}$. In fact, the operator $V \stackrel{T}{\lra} V$ realizing this isomorphism is exactly $\alpha_l(\xi_l^i)$, and since $V = \Ind^{\GL{n-1}{k_E}}_{\mathrm U_{n-1}(k_E)} \psi$, we have (denoting $Y = \left( \mathrm U_{n-1}(k_E) \backslash \GL{n-1}{k_E} \right)$ \[ T^0 \left( \alpha_l(\xi_l^i), V \right)  \cong \Gamma_c \left( Y^{c_{ \xi_l^i }} , T^0 ( \mathcal F' ) \right) \] where $\mathcal F'$ is the $l$-sheaf associated to $V$, that is \[ \mathcal F' (A) = \begin{array}{l} \{ f: \GL{n-1}{k_E} \lra \overline \Zl \textnormal{ supported on } \mathrm U_{n-1}(k_E) A \\ \textnormal{ and such that } f(ug) = \psi(u) f(g) \, \forall u \in \mathrm U_{n-1}(k_E), \, g \in \GL{n-1}{k_E} \}. \end{array} \]
Under this notation, the map of sheaves making $\mathcal F'$ a $c_{\xi_l^i}$-equivariant sheaf is \[ c_{\xi_l^i} (A): \mathcal F' (A) \lra \mathcal F' (c_{\xi_l^i}^{-1} (A) ) \quad \left( c_{\xi_l^i} (A).f \right) (g) = f(c_{\xi_l^i} (g)), \] so we will compute Tate cohomology of $\mathcal F'$ with respect to this map.
By the usual argument about automorphisms of prime order $l$ fixing a subgroup of order coprime to $l$, we have that any coset $\mathrm U_{n-1}(k_E) g \in Y$ fixed by $c_{\xi_l^i}$ admits a representative which is itself fixed, i.e. we can assume that $c_{\xi_l^i}(g) = g$ for each $\mathrm U_{n-1}(k_E) g \in Y^{c_{\xi_l^i}}$. \\
Fix then such a coset, we want to compute \[ T^0(\mathcal F') (\mathrm U_{n-1}(k_E) g) = \ker \left( 1 - \alpha_l(\xi_l^i) \right) \left( \mathrm U_{n-1}(k_E) g \right) / \im N(\mathrm U_{n-1}(k_E) g ). \] Any $f \in \ker \left( 1 - \alpha_l(\xi_l^i) \right) \left( \mathrm U_{n-1}(k_E) g \right) $ will satisfy $f(h) = f(c_{\xi_l^i} (h))$. Choose some $1 \le k \le n-1$ such that $i_k  > i_{k+1}$ and denote by $e_{k, k+1}$ the elementary matrix with a $1$ in the $(k, k+1)$ slot and zeros everywhere else; we have \[ c_{\xi_l^i} \left( I_n + e_{k,k+1} \right) = I_n + \xi_l^{i_k - i_{k+1}} e_{k, k+1} \] and hence choosing $h = \left( I_n + e_{k,k+1} \right) g$ we obtain \[ \psi(1) f(g) = f(h) = f(c_{\xi_l^i}(h)) = f( I_n + \xi_l^{i_k - i_{k+1}} e_{k, k+1} g) = \psi(\xi_l^{i_k- i_{k+1}}) f(g) \] which implies $f(g) = 0$, because the regular character $\psi$ can be chosen to be injective, and hence $\psi(1) \neq \psi(\xi_l^{i_k - i_{k+1}})$.

This shows that $ \ker \left( (1 - \alpha_l(\xi_l^i) \left( \mathrm U_{n-1}(k_E) g \right) \right) = 0$ for each $\mathrm U_{n-1}(k_E) g \in Y^{c_{\xi_l^i}}$, and in particular \[ \Gamma_c \left( Y^{c_{\xi_l^i}} , T^0 (\mathcal F') \right) = 0 \] which proves the proposition. 

\end{proof}

We are finally ready to prove theorem \ref{ramcase}.
\begin{proof}[Proof of theorem \ref{ramcase}] Both $r_l(\pmb \pi)$ and $T^0 (\pmb \rho)$ are now $l$-modular representations that can be viewed as $l$-sheaves on the vertices of the Bruhat-Tits building of $\GL{n}{E}$, therefore it suffices to check that at each vertex they assign the same vector space, with compatible actions.

Fix $w = E^* \GL{n}{\OO_E}g_w \in \mathcal C^0 \left( \mathcal B(\GL{n}{E} \right) $ for some $g_w \in \GL{n}{E}$ and consider \[ \Gamma_c \left( w , \mathcal F \right) = \begin{array}{l} \{ f: \GL{n}{F} \lra  V  \textnormal{ supported on $E^* \GL{n}{\OO_E} g_w$ and such that} \\ f(zkg) = \omega_{\alpha_l}(z) \alpha_l(k).f(g) \quad \forall z \in F^*, \, k \in \GL{n}{\OO_F}, \, g \in \GL{n}{F} \}. \end{array}. \]
Since $g_w$ is $\sigma$-fixed, every $f \in \Gamma_c \left( w , \mathcal F \right)$ is trivially in $\ker (1- \lambda)(w)$: equivalently, $f(g_w) = v$ can be any vector in $V$. Now $(Nf)(g) = l \cdot f(g)$, hence we obtain that \[ \Gamma_c \left( w , T^0 ( \mathcal F ) \right) = \begin{array}{l} \{ f: \GL{n}{F} \lra  V \otimes \overline \Fl  \textnormal{ supported on $E^* \GL{n}{\OO_E} g_w$ and such that} \\ f(zkg) = \omega_{\alpha_l}(z) \alpha_l(k).f(g) \quad \forall z \in F^*, \, k \in \GL{n}{\OO_F}, \, g \in \GL{n}{F} \}. \end{array}. \]
As in proposition \ref{compat}, (1), the reduction of the lift $r_l( \pmb \pi)$ is isomorphic to the lift of the reduction, hence \[ \Gamma_c \left( w , r_l( \pmb \pi) \right) = \begin{array}{l} \{ f: \GL{n}{E} \lra  V \otimes \overline \Fl \textnormal{ supported on $E^* \GL{n}{\OO_E} g_w$ and such that} \\ f(zkg) = \omega_{\alpha}(z) \alpha(k).f(g) \quad \forall z \in F^*, \, k \in \GL{n}{\OO_F}, \, g \in \GL{n}{F} \}. \end{array}. \]
Notice now that after reducing $\mod l$, we have \[ r_l(\alpha)^{(l)} = r_l(\alpha) \otimes_{ ( \overline \Fl, \Frob_l)} \overline \Fl \cong r_l(\alpha_l) \] where the tensor product on the left is the Frobenius twist, because for the corresponding Green characters we have \[ \chi \otimes_{\overline \Fl} \overline \Fl = \chi^l \] and hence the Brauer characters of $r_l(\alpha)^{(l)}$ and $ r_l(\alpha_l)$ coincide. Both modular representations are naturally realized on $V \otimes \overline \Fl$. \\
Thus at every vertex $w \in \mathcal \mathcal C^0 \left( \mathcal B(\GL{n}{E} \right)$ the two $l$-sheaves $r_l(\pmb \pi)^{(l)}$ and $T^0 (\pmb \rho)$ have the same space of sections (as in proposition \ref{compat}, (2), taking Frobenius twists commutes with lifting) and the actions of $\Stab_{\GL{n}{E}}(w)$ coincide. Since the translation action of $\GL{n}{E}$ on the vertices $\mathcal C^0 \left( \mathcal B(\GL{n}{E} \right)$ is also obviously the same, the two $l$-sheaves $r_l(\pmb \pi)^{(l)}$ and $T^0(\pmb \rho)$ coincide, and by the Bernstein-Zelevinsky' equivalence of categories this proves that the two representations $r_l(\pmb \pi)^{(l)}$ and $T^0(\pmb \rho)$ are isomorphic.
\end{proof}


\begin{thebibliography}{aaa}

\bibitem{AC} Arthur - Clozel,
\textit{Simple Algebras, Base Change, and the Advanced Theory of the Trace Formula},
Annals of Mathematical Studies, Princeton University Press, number 120, 1989.

\bibitem{arthur} Arthur,
\textit{The principle of functoriality},
Bulletin of the AMS, vol.40, no.1, 39-53, 2002.

\bibitem{bernstein} Bernstein, Rummelhart,
\textit{Draft of: Representations of $p$-adic groups},
lectures at Harvard University, 1992.

\bibitem{BZ1} Bernstein, Zelevinsky,
\textit{ Induced representations of reductive $p$-adic groups, I},
Annales Scientifiques de l'E.N.S. ,1977.

\bibitem{BZ2} Bernstein, Zelevinsky,
\textit{ Representations of the group $\GL{n}{F}$ where $F$ is a non-archimedean local field},
Russian Math. Surveys 31:3 (1976).

\bibitem{Borchards1} Borcherds,
\textit{Modular Moonshine III},
Duke Mathematical Journal 93 (1998), no. 1, 129-154.

\bibitem{bump} Bump,
\textit{Automorphic forms and representations},
Cambridge Studies in Advanced Mathematics, volume 55, 1997.

\bibitem{BH} Bushnell, Henniart,
\textit{Modular Local Langlands Correspondence for $\Gl{n}$},
Int Math Res Notices (2014) 2014 (15): 4124-4145.

\bibitem{BH1} Bushnell, Henniart,
\textit{The essentially tame Local Langlands Correspondence I},
Journal of the AMS, volume 18, number 3, 2005.

\bibitem{BH2} Bushnell, Henniart,
\textit{The essentially tame Local Langlands Correspondence III, the general case},
Proceedings of the London Math. Soc. (3), 101, 2010.

\bibitem{C} Carter,
\textit{Finite groups of Lie type - conjugacy classes and complex characters},
John Wiley and Sons, 1985.

\bibitem{CR} Curtis, Reiner,
\textit{Methods of representation theory, volume I},
John Wiley and Sons, 1981.

\bibitem{FV} Fesenko, Vostokov,
\textit{Local Fields and their extensions - 2nd edition},
Translations of Mathematical Monographs, vol. 121, AMS, 2002.

\bibitem{gelbart} Gelbart,
\textit{An elementary introduction to the Langlands program}
Bulletin of the AMS, vol. 10, no.2, 177-219, 1984.

\bibitem{Gelfand} Gelfand,
\textit{Representations of the Full Linear Group over a Finite Field},
Math USSR Sbornik, vol. 12 (1970), no. 1.

\bibitem{G} Green,
\textit{The characters of the finite general linear groups},
Transaction of the AMS, vol. 80, no. 2, 1955.

\bibitem{M} Macdonald,
\textit{Symmetric functions and Hall polynomials},
Oxford University Press, 1978.

\bibitem{Prasad} G. Prasad,
\textit{Galois-fixed points in the Bruhat-Tits building of a reductive group},
Bull. Soc. math. France, 129 (2), 2001, p. 169-174.

\bibitem{renard} Renard,
\textit{Representations des groupes reductifs $p$-adiques},
Cours Specialises SMF, 2010.

\bibitem{serre} Serre,
\textit{Linear Representation of Finite Groups},
Graduate texts in Mathematics, volume 42, 1977.

\bibitem{serre2} Serre,
\textit{Local Fields},
Graduate texts in Mathematics, volume 67, 1980.

\bibitem{Sh} Shintani,
\textit{Two remarks on irreducible characters of finite general linear groups},
Journal of the Mathematical Society of Japan, vol. 28, no. 2, 1976.

\bibitem{S1} Springer,
\textit{Cusp forms for finite groups},
Seminar on algebraic groups and related finite groups, Lecture Notes in Mathematics vol.131, Springer, 1970.

\bibitem{S2} Springer,
\textit{Characters of special groups},
Seminar on algebraic groups and related finite groups, Lecture Notes in Mathematics vol.131, Springer, 1970.

\bibitem{SS} Springer, Steinberg,
\textit{Conjugacy classes},
Seminar on algebraic groups and related finite groups, Lecture Notes in Mathematics vol.131, Springer, 1970.

\bibitem{tate} Tate,
\textit{Number Theoretic Background},
Proceedings of Symposia in Pure Mathematics, volume 33, part 2, 1979.

\bibitem{TV} Treumann, Venkatesh,
\textit{Functoriality, Smith Theory and the Brauer Homomorphism},
http://arxiv.org/pdf/1407.2346.pdf.

\bibitem{Vig} Vign\'eras,
\textit{Correspondance de Langlands semi-simple pour $\mathrm{GL}(n,F)$ modulo $l \neq p$},
Inventiones mathematicae vol. 144, 177-223, 2001.

\bibitem{V} Vign\'eras,
\textit{Repr\'esentations $l$-modulaires d'un group r\'eductif $p$-adiques avec $l \neq p$},
Birkhauser, 1996.

\bibitem{vogan} Vogan,
\textit{The local Langlands conjecture},
Representation theory of groups and algebras, Contemp. Math. 145, AMS, pp. 305-379.

\end{thebibliography}
\end{document}